\documentclass[11pt]{preprint}
\usepackage[full]{textcomp}
\usepackage[osf]{newtxtext} 
\usepackage[cal=boondoxo]{mathalfa}
\usepackage{colortbl}
\usepackage[normalem]{ulem}
\usepackage{amsmath}
\usepackage{pgf}
\usepackage{xcolor}
\usepackage{comment}

\usepackage{amssymb}
\usepackage{mathtools}
\usepackage{hyperref}
\usepackage{breakurl}
\usepackage{mhenvs}
\usepackage{mhequ} 
\usepackage{mhsymb}
\usepackage{booktabs}
\usepackage{tikz}
\usepackage{tcolorbox}
\usepackage{mathrsfs}
\usepackage[utf8]{inputenc}
\usepackage{longtable}
\usepackage{wrapfig}

\usepackage{microtype}
\usepackage{comment}
\usepackage{wasysym}
\usepackage{centernot}
\usepackage{enumitem}
\usepackage{bm}
\usepackage{stackrel}
\usepackage{graphicx}

\makeatletter
\newcommand{\globalcolor}[1]{%
	\color{#1}\global\let\default@color\current@color
}
\makeatother

\usetikzlibrary{calc}
\usetikzlibrary{decorations}
\usetikzlibrary{positioning}
\usetikzlibrary{shapes}
\usetikzlibrary{external}

\newif\ifdark
\darkfalse

\ifdark

\definecolor{darkred}{rgb}{0.9,0.2,0.2}
\definecolor{darkblue}{rgb}{0.7,0.3,1}
\definecolor{darkgreen}{rgb}{0.1,0.9,0.1}
\definecolor{franck}{rgb}{0,0.8,1}
\definecolor{pagebackground}{rgb}{.15,.21,.18}
\definecolor{pageforeground}{rgb}{.84,.84,.85}
\pagecolor{pagebackground}
\AtBeginDocument{\globalcolor{pageforeground}}
\definecolor{symbols}{rgb}{0,0.7,1}
\colorlet{connection}{red!80!black}
\colorlet{boxcolor}{blue!50}

\else

\definecolor{darkred}{rgb}{0.7,0.1,0.1}
\definecolor{darkblue}{rgb}{0.4,0.1,0.8}
\definecolor{darkgreen}{rgb}{0.1,0.7,0.1}
\definecolor{franck}{rgb}{0,0,1}
\definecolor{pagebackground}{rgb}{1,1,1}
\definecolor{pageforeground}{rgb}{0,0,0}
\colorlet{symbols}{blue!90!black}
\colorlet{connection}{red!30!black}
\colorlet{boxcolor}{blue!50!black}

\fi

\def\slash{\leavevmode\unskip\kern0.18em/\penalty\exhyphenpenalty\kern0.18em}
\def\dash{\leavevmode\unskip\kern0.18em--\penalty\exhyphenpenalty\kern0.18em}

\DeclareMathAlphabet{\mathbbm}{U}{bbm}{m}{n}

\DeclareFontFamily{U}{BOONDOX-calo}{\skewchar\font=45 }
\DeclareFontShape{U}{BOONDOX-calo}{m}{n}{
	<-> s*[1.05] BOONDOX-r-calo}{}
\DeclareFontShape{U}{BOONDOX-calo}{b}{n}{
	<-> s*[1.05] BOONDOX-b-calo}{}
\DeclareMathAlphabet{\mcb}{U}{BOONDOX-calo}{m}{n}
\SetMathAlphabet{\mcb}{bold}{U}{BOONDOX-calo}{b}{n}

\setlist{noitemsep,topsep=4pt,leftmargin=1.5em}

\DeclareMathAlphabet{\mathbbm}{U}{bbm}{m}{n}

\DeclareMathAlphabet{\mcb}{U}{BOONDOX-calo}{m}{n}
\SetMathAlphabet{\mcb}{bold}{U}{BOONDOX-calo}{b}{n}
\DeclareFontFamily{U}{mathx}{\hyphenchar\font45}
\DeclareFontShape{U}{mathx}{m}{n}{
	<5> <6> <7> <8> <9> <10>
	<10.95> <12> <14.4> <17.28> <20.74> <24.88>
	mathx10
}{}
\DeclareSymbolFont{mathx}{U}{mathx}{m}{n}
\DeclareMathSymbol{\bigtimes}{1}{mathx}{"91}

\providecommand{\figures}{false}
{ \ifthenelse{\equal{\figures}{false}} {#1}{\[ {\rm Figure \ missing !} \]} }{}

\newtheorem{scheme}{Scheme}[section]

\definecolor{oxfordblue}{rgb}{0.0, 0.13, 0.28}
\tikzstyle{leaf}=[circle, draw=black, fill=gray!20, inner sep = .1pt, font=\tiny, minimum size=3mm]
\tikzstyle{inner}=[circle, minimum size=1.5mm, fill=oxfordblue, inner sep =0]
\tikzstyle{conj} = [thick, color=maroon, dotted,dash pattern=on 1pt off 1pt]
\tikzstyle{conj-int} = [color=cyan, dotted]
\tikzstyle{int} = [color=cyan]
\tikzstyle{positive} = [thick, color=maroon]
\tikzstyle{Phi} = [thick, color=maroon, snake=snake,
segment length=5pt,
segment amplitude=1pt]
\tikzstyle{Phi-conj} = [thick, color=maroon, snake=snake,
segment length=5pt,
segment amplitude=1pt,dashed]
\tikzstyle{stoch-int} = [color=cyan, snake=zigzag,
segment amplitude=.2mm,
segment length=.5mm]
\tikzstyle{stoch-conj-int} = [color=cyan, snake=triangles,
segment amplitude=.25mm,
segment length=.8mm]
\definecolor{modebeige}{rgb}{0.59, 0.44, 0.09}
\definecolor{maroon}{rgb}{0.5, 0.0, 0.0}
\definecolor{mayablue}{rgb}{0.45, 0.76, 0.98}

\newcommand{\msL}{\mathscr{L}}

\newcommand{\cF}{{\mathcal F}}

\def\CH{\mathcal{H}}

\def\CT{\mathcal{T}}
\def\mcO{\mathcal{O}}

\tikzstyle{tinydots}=[dash pattern=on \pgflinewidth off \pgflinewidth]
\tikzstyle{superdense}=[dash pattern=on 4pt off 1pt]

\newcommand{\mcL}{\mathcal{L}}

\newcommand{\mcT}{\mathcal{T}}

\newcommand{\mbbE}{\mathbb{E}}

\newcommand{\mbbN}{\mathbb{N}}
\newcommand{\mbbP}{\mathbb{P}}
\newcommand{\RR}{\mathbb{R}}

\newcommand{\mbbT}{\mathbb{T}}

\newcommand{\mbbZ}{\mathbb{Z}}

\newcommand{\mfn}{\mathfrak{n}}

\newcommand{\mfe}{\mathfrak{e}}

\newcommand{\mfL}{\mathfrak{L}}

\newcommand{\mft}{\mathfrak{t}}

\newcommand{\mfp}{\mathfrak{p}}

\newcommand{\mfl}{\mathfrak{l}}

\newcommand{\mff}{\mathfrak{f}}

\def\Lab{\mathfrak{L}}
\newcommand{\Labhom}{\mathfrak{t}}

\def\${|\!|\!|}

\def\msL{\mathscr{L}}

\def\br{\penalty0}

\newcommand{\pa}{\partial}
\newcommand{\rbr}[1]{\left(#1\right)}
\newcommand{\sqbr}[1]{\left[#1\right]}
\newcommand{\abs}[1]{\left| #1 \right|}
\newcommand{\norm}[1]{\left| \left| #1 \right| \right|}

\newcommand{\set}[1]{\left\lbrace  #1 \right\rbrace }

\newcommand{\nin}{\backslash}

\newcommand{\Ups}{\Upsilon}
\newcommand{\lam}{\lambda}

\newcommand{\grad}{\nabla}

\newcommand{\lap}{\Delta}

\newcommand{\pint}[1]{\the\numexpr #1 \relax}

\newenvironment{DIFnomarkup}{}{} 

\theorembodyfont{\rmfamily}

\newfont{\indic}{bbmss12}

\def\Nabla_#1{\nabla_{\!#1}}

\let\eps\varepsilon

\def\eqref#1{(\ref{#1})}

\makeatletter 
\newcommand*{\bigcdot}{}
\DeclareRobustCommand*{\bigcdot}{%
	\mathbin{\mathpalette\bigcdot@{}}%
}
\newcommand*{\bigcdot@scalefactor}{.5}
\newcommand*{\bigcdot@widthfactor}{1.15}
\newcommand*{\bigcdot@}[2]{%
	\sbox0{$#1\vcenter{}$}
	\sbox2{$#1\cdot\m@th$}%
	\hbox to \bigcdot@widthfactor\wd2{%
		\hfil
		\raise\ht0\hbox{%
			\scalebox{\bigcdot@scalefactor}{%
				\lower\ht0\hbox{$#1\bullet\m@th$}%
			}%
		}%
		\hfil
	}%
}
\makeatother

\tcbset
{colframe=boxcolor,colback=symbols!7!pagebackground,coltext=pageforeground,
	fonttitle=\bfseries,nobeforeafter,center title,size=fbox,boxsep=1.5pt,
	top=0mm,bottom=0mm,boxsep=0mm,tcbox raise base}

\def\two{{\<generic>\kern0.05em\<genericb>}}
\def\twoI{{\<Ito>\kern0.05em\<Itob>}}

\def\mail#1{\burlalt{#1}{mailto:#1}}

\begin{document}
	\usetikzlibrary{snakes}
	\title{Resonance based schemes for SPDEs}\author{ J. Armstrong-Goodall$^1$, Y. Bruned$^2$}
	\institute{ Maxwell Institute for Mathematical Sciences, University of Edinburgh \and  IECL (UMR 7502), Université de Lorraine
		\\
		Email:\ \begin{minipage}[t]{\linewidth}
		\mail{j.a.armstrong-goodall@sms.ed.ac.uk}, \\ 	\mail{yvain.bruned@univ-lorraine.fr}.
	\end{minipage}}
	
	\maketitle
	
	\begin{abstract}
		Resonance based numerical schemes are those in which cancellations in the oscillatory components of the equation are taken advantage of in order to reduce the regularity required of the initial data to achieve a particular order of error and convergence. We investigate the potential for the derivation these schemes in the context of nonlinear stochastic PDEs. By comparing the regularity conditions required for error analysis to traditional exponential integrators we demonstrate that at orders less than $ \mcO(t^2) $, that resonance based techniques are successful and provide a significant gain on the regularity of the initial data while, at orders greater than $ \mcO(t^2) $, they do not achieve any gain. This restriction is due to limitations in the explicit path-wise analysis of stochastic integrals. As examples of applications of the method, we present schemes for the Schr\"odinger equation and Manakov system accompanied by local error and stability analysis, as well as proof of global convergence in both the strong and path-wise sense.  
	\end{abstract}

	\setcounter{tocdepth}{2}
	\tableofcontents

	\section{Introduction}
	In this work, we are investigating time discretisation for a large class of stochastic partial differential equations which are of the form:
	\begin{equation}\label{dis}
		\begin{aligned}
			& i \partial_t u(t,x) +   \mathcal{L}\left(\nabla \right) u(t,x) =\vert \nabla\vert^\alpha p\left(u(t,x), \overline u(t,x)\right) \\ & +\vert \nabla\vert^\beta  f(u(t,x),\overline u(t,x)) W(t,x), \quad
			u(0,x) = v(x),
		\end{aligned}
	\end{equation}
	where we focus on the case of periodic boundary conditions $x\in \mathbb{T}^d$. Further, we assume that $p$ is a polynomial nonlinearity and $ f $ is linear in $ u(t,x) $ or  $ \overline{u}(t,x) $. The term $ W(t,x) $ is a space-time noise. We assume that structure of \eqref{dis} implies at least almost local wellposedness of the problem on a finite time interval $]0,T]$, $T<\infty$, in an appropriate functional space. In general, we will work with a noise white in time and coloured in space. 
	
	A general framework of resonance based schemes for this class of equations was developed in \cite{B.S2022} for the deterministic setting, that is, without noise ($f=0$). The key  idea of \cite{B.S2022} lies in embedding the underlying structure of resonances - triggered by the nonlinear frequency interactions between the leading differential operator $\mathcal{L}\left(\nabla \right)$ and the nonlinearities $p\left(u(t,x), \overline u(t,x)\right)$ and $ \vert \nabla\vert^\beta  f(u(t,x),\overline u(t,x)) W(t,x) $ - into the numerical discretisation. These nonlinear interactions are in general neglected by classical approximation techniques such as Runge-Kutta methods, splitting methods or exponential integrators. While for smooth solutions these nonlinear interactions are indeed negligible, they do play a central role at low regularity. This crucial idea was intially worked out by constructing specific examples:
	The first {\it resonance based schemes} were developed on the torus for the Korteweg de Vries (KdV) equation \cite{H.S2017} 
	\begin{equation}\label{kdvIntro}
		\partial_t u -i\mathcal{L}\left(\nabla\right) u = \frac12 \partial_x u^2, \quad \mathcal{L}\left(\nabla\right) = i\partial_x^3, \quad \vert \nabla\vert^\alpha = \partial_x, \quad x \in \mathbb{T},
	\end{equation}
	and later for the nonlinear Schr\"odinger (NLS) equation \cite{O.S2018},
	\begin{equation}\label{nlsIntro}
		i \partial_t u + \mathcal{L}\left(\nabla\right)  u = \vert u\vert^2 u, \quad \mathcal{L}\left(\nabla\right) = \Delta, \quad x\in \mathbb{T}^d.
	\end{equation}
	The main contribution of \cite{B.S2022} is to systematize this approach via the use of decorated trees that combine two combinatorial structures. The first combinatorial structure are decorated trees that reflects the Fourier decomposition of an iterated integral, know for a while in dispersive equations (see for example \cite{C2009,G2012,G.K.O2013}). The second combinatorial structure are decorated trees used for  singular stochastic partial differential equations. These trees have provided a general framework for solving a large class of singular SPDEs via the theory of Regularity Structures (see \cite{B.C.C.H2020,B.H.Z2019,C.H2018,H2014}). The decorated trees in \cite{B.S2022} are used for coding a new type of Butcher series extending \cite{B1972} and together with Hopf algebraic structures, one can conduct the local error analysis.  These low regularity schemes show better behaviour and improve the regularity normally required in classical numerical schemes (see \cite{H.O2010,H.K.L.R2010,I.Z2009,L2008}).
	
	Several extensions have been considered within this general framework such as: more general nonlinearities like non-polynomial $ f $ and a deterministic spatial potential $ W(x) $ (see \cite{B.B.S2022}), in Wave turbulence theory when one wants to discretise the second moment of the Fourier coefficient of the solution (see \cite{B.B.S2023}) and the construction of symmetric schemes by exploiting different ways of iterating Duhamel formulation and by using symmetric polynomial interpolation (see \cite{B.B.M.S2023}). Symplectic methods are partially treated \cite{M.S2023} and are still an open problem for low regularity schemes.
	
	In the present work, we investigate an extension of this framework to SPDEs but not with a singular noise as some smoothness is required in space. The idea is to apply also the resonance analysis in this context for minimising the regularity asked on the initial data. Indeed, in the literature for time step integrator for SPDEs, the regularity asked on the initial data is pretty large and most of the time low order schemes are considered. Previous work on numerical schemes for the stochastic NLS (SNLS) with both additive and multiplicative noise is done by De Bouard and Debussche in \cite{B.D2004} and \cite{B.D2006} where they present error analysis for the stochastic version of the energy preserving scheme introduced in the earlier paper by Delfour, Fortin and Payre in \cite{D.F.P1981}. The Manakov equation is presented in the context of optical physics in \cite{M.A.S2012} and the mathematical analysis performed in the thesis of Gazzeau \cite{G2012a} and the paper by the same author on a numerical scheme \cite{G2014}. This is followed by work done by Berg, Cohen and Dujardin in \cite{B.C.D2020,B.C.D2021}. 
	
	Our main contribution is to show that for low order schemes it is possible to get a significant improvement for the regularity asked on the initial data. Our strategy is to apply classical discretisations to integrals containing only the noise and to use the resonance analysis developed in \cite{B.S2022} for dealing with the nonlinearities.
	We apply our techniques to the stochastic NLS with additive and multiplicative noise 
	\begin{equation}
		i\pa_t u(t,x) + \lap u(t,x) + \lam u(t,x) \abs{u(t,x)}^2  = \sigma(u(t,x)))  W(t,x)
	\end{equation}
	where $ \sigma(u) \in\set{\one, u(t,x)} $ and $ x \in  \mbbT^3  $. The Manakov system, which in Stratanovich form reads,
	\begin{equs}
		i \partial_t u(t,x) + \Delta u(t,x) + \lam u(t,x) \abs{u(t,x)}^2  + i\sqrt{\gamma}\sum_{j=1}^3\sigma_k\grad u(t,x) \circ dW_j(t) = 0
	\end{equs}
	In the Manakov system the noises are now three independent Brownian motions, the $ \sigma_i $ refer to the Pauli matrices and the gradient structure of the multiplicative noise turns out to ask some regularity on the initial data.
	We proceed by first iterating Duhamel's formula, see \eqref{dis}, in Fourier space, producing a decorated tree series that encodes various iterated integrals.
	This is given by a Butcher's series with a more combinatorially complex set of decorated trees, 
	\begin{equs}
		U^{r}_k(t,v) =   \sum_{T \in \CT^{r,k}_{0}(R)} \frac{\Upsilon( T)(v)}{S(T)} (\Pi   T )(t),
	\end{equs}
	where $ k \in \mathbb{Z}^d $, $ \CT^{r,k}_{0}(R) $ is a set of decorated trees of order smaller than $ \mathcal{O}(t^{r}) $. This implies that the iterated integral $ (\Pi   T )(t) $ associated to the tree $T$ has order less than: $ (\Pi   T )(t) = \mathcal{O}(t^{\alpha_{\tau}}) $, with $ \alpha_{\tau} \leq r $.  The coefficients $ \Upsilon(T) $ and $ S(T) $ are, respectively, the elementary differential and the symmetry factor associated to $T$. From Proposition~\ref{tree_series}, it follows that this series is a good approximation of the $k^{th}$ Fourier coefficient $ u_k $ of the solution with order $ \mathcal{O}(t^{r + \frac{1}{2}}) $. The fractional power $ \frac{1}{2} $ comes from the fact that we consider Wiener process in time. The general scheme is obtained by replacing each of the $ (\Pi T)(t) $ by the corresponding discretisation $ (\Pi^{n,r} T)(t) $ where $ n $ is the regularity of the initial data $ v \in H^n $ and $ r $ is the order of the discretisation, in the sense that $  (\Pi^{n,r} T)(t) $ is of order $ \mathcal{O}(t^{r + \frac{1}{2}}) $. Consequently, the scheme is given by
	\begin{equs}
		U^{n,r}_k(t,v) =   \sum_{T \in \CT^{r,k}_{0}(R)} \frac{\Upsilon( T)(v)}{S(T)} (\Pi^{n,r}   T )(t).
	\end{equs}
	The low regularity paradigm relies on the premise that the discretisation of $ (\Pi^{n,r}   T )(t) $ can be done in such a way that the terms are easy to write back in Physical space with usual operators. The discretisation can be divided into two different types of integral:
	\begin{equs}
		\text{Resonance: }  \int_{0}^{t}e^{isP}d s.
		\quad
		\text{Stochastic: }  \int_{0}^{t}e^{isP} \Phi_{k_i} d W_{k_i}(s).
	\end{equs}
	where $ P $ in both cases is polynomial in some frequencies $k_1,...,k_n$ and $ \Phi_{k_i} d W_{k_i}(s) $ is the $ k_i $th Fourier coefficient of $ \Phi d W(s) $. For the resonance integral, one wants to decompose $ P $ into lower part $P_{\text{\tiny{low}}}$ and dominant part $ P_{\text{\tiny{dom}}} $ such that:
	\begin{equs}
		P = P_{\text{\tiny{dom}}} + P_{\text{\tiny{low}}}, \quad \int_{0}^{t}e^{isP}d s = \int_{0}^{t}e^{isP_{\text{\tiny{dom}}}}  \left( 1 + i s P_{\text{\tiny{low}}} + .... \right) d s.
	\end{equs}
	Then the dominant part is integrated exactly to produce an approximation at order $ \mathcal{O}(t^2) $ of the form:
	\begin{equs}
		\int_{0}^{t}e^{isP}d s = \frac{e^{it P_{\text{\tiny{dom}}}}-1}{iP_{\text{\tiny{dom}}} } + \mathcal{O}(t^2 P_{\text{\tiny{low}}}) 
	\end{equs}	
	In order to guarantee the success of this approach one has to have $ P_{\text{\tiny{low}}} $ of lower degree than $ P $ in the sense that it will ask less regularity on the initial data. The second crucial point is to be able to rewrite $ \frac{1}{P_{\text{\tiny{dom}}}} $ in Physical space with usual differential operators. This technique works perfectly well for deterministic integrals when such a good decomposition into dominant and lower part exists which is the case for some dispersive equations.
	For the stochastic integral, resonance analysis cannot be performed. Indeed, one cannot have in general a close form for
	\begin{equs}
		\int_{0}^{t}e^{isP_{\text{\tiny{dom}}}} \Phi_{k_i} d W_{k_i}(s).
	\end{equs}
	In this case, we are forced to use classical schemes and fully Taylor expand:
	\begin{equs}
		\int_{0}^{t}e^{isP_{\text{\tiny{dom}}}} \Phi_{k_i} d W_{k_i}(s) = \int_{0}^{t} \Phi_{k_i} d W_{k_i}(s) + \mathcal{O}(t^{\frac{3}{2}}P_{\text{\tiny{dom}}} \Phi_{k_i}).
	\end{equs}
	This will require some regularity on $ \Phi $ for absorbing the extra regularity needed. What we observe in the end is that we expect to have gain on the regularity on the initial data up to the order $ \mathcal{O}(t^{3/2}) $ and for higher order we cannot do better than usual schemes. This is summarised in Theorem~\ref{main_theorem_general} the main theorem of the paper. We check the first part on examples in the sequel by looking at the NLS with multiplicative noise and the Manakov system but its general proof is quite similar to the one for NLS which contains most of the difficulties. The main results are given by the Schemes \ref{scheme:schrodinger_O1/2}, \ref{scheme:schrodinger_O1} \ref{scheme:manakov_O1/2} and their local error are treated in  Propositions \ref{expansion_scheme_NLS}, \ref{prop:schrodinger_O1/2_local}, \ref{Manakov_local_error} and Theorem~\ref{scheme_order_2}. Satibility and global error are proved in Propositions \ref{lem:stability} and \ref{prop:global_error} For the second part of the theorem which says that one cannot do better than usual schemes for higher orders, we exhibit an example of iterated integrals where low regularity cannot be achieved. This is the object of Proposition~\ref{counter_example}. 
	
	One future direction will be to look to schemes satisfying some symmetries. Symmetric schemes, those in \cite{B.B.M.S2023} for instance, are well understood in the determinist context. One can then easily build schemes of order two that will be symmetric. The main task will be to check how the stochastic methods will work with implicit schemes. Another interesting symmetry will be symplectic methods (see  \cite{M.S2023}) but they are not very general for deterministic schemes. An interesting but technically challenging question is whether it's possible to produce schemes with both low regularity spatial noise and initial conditions. 
	
	Finally, let us outline the paper by summarising the content of its sections. In Section~\ref{Sec::2}, we present the main SPDEs under consideration and precisely define the noise. We start  here the first step of the scheme by explaining how to iterate the Duhamel formulation in order to produce a series of decorated trees that will be truncated depending on the order of the scheme. Decorated trees are introduced in Definition \ref{decorated_trees} and the truncated tree series is given in Proposition~\ref{tree_series}. At the end of this section, we state one of our main theorems, that is, Theorem~\ref{main_theorem_general}. It  gives the order at which the resonance based approach ceases to work for SPDEs.
	
	In Section~\ref{Sec::3}, we introduce the technique of resonance analysis which is the main tool for discretising the iterated integrals. We first explain the discretisation of the nonlinearity before comparing this to the discretisation of the stochastic integral. The main example to which we apply these techniques is the stochastic NLS with multiplicative noise and cubic nonlinearity for which we derive a discretisation in Fourier space up to order $ \mathcal{O}(t^{3/2}) $ in Propositon \ref{expansion_scheme_NLS}. We also discuss the mechanism by which we can lower regularity conditions in the deterministic case and how this does not translate to stochastic PDEs at higher orders. The latter is verified by exhibiting an iterated integral where the resonance techniques cannot be applied in Proposition \ref{counter_example}. 
	
	In Section~\ref{Sec::4}, we begin by computing the local error analysis of various schemes for SPDEs including stochastic NLS with additive and multiplicative noise and the Manakov equation. The main schemes are given by Schemes \ref{scheme:schrodinger_O1/2}, \ref{scheme:schrodinger_O1} \ref{scheme:manakov_O1/2} and their local error are treated in  Propositions \ref{expansion_scheme_NLS}, \ref{prop:schrodinger_O1/2_local}, \ref{Manakov_local_error} and Theorem~\ref{scheme_order_2}. We introduce methods previously used in PDEs for stabilizing numerical schemes and proving global convergence and apply them to the analysis of schemes for SPDEs. These methods lead to a streamlined method for establishing convergence since we can avoid an implicit discretisation of the noise. The method of obtaining global error estimated is quite general as it is derived directly from the local error using arguments originally established for ODEs and SDEs, see respectively \cite{B.B2001,H.W.L2002}.  The main results in this section are Propositions \ref{lem:stability} and \ref{prop:global_error}.

	\subsection*{Acknowledgements}
	
	{\small
		Y. B. thanks the Max Planck Institute for Mathematics in the Sciences (MiS) in Leipzig for having supported his research via a long stay in Leipzig from January to June 2022. Y. B. gratefully acknowledges funding support from the European Research Council (ERC) through the ERC Starting Grant Low Regularity Dynamics via Decorated Trees (LoRDeT), grant agreement No.\ 101075208.  J. A. G.  thanks the Max Planck Institute for Mathematics in the Sciences (MiS) for a short stay in Leipzig during which this work started. 
		J. A. G. is supported by the EPSRC Centre for Doctoral Training in Mathematical Modelling, Analysis and Computation (MAC-MIGS) funded by the UK Engineering and Physical Sciences Research Council (grant EP/S023291/1), Heriot-Watt University and the University of Edinburgh.
	}

	\section{Duhamel iteration and decorated trees}
	\label{Sec::2}
	We introduce the equations in more detail in \ref{dispersive_spde_general}, specifying the definitions for each term. Most importantly we carefully define the cylindrical Brownian motion and the Hilbert-Schmidt operator which acts by smoothing out the noise in space. It's also necessary to introduce Duhamel's formulation or mild form of each SPDE which are the object we wish to discretise in order to produce our schemes. Related to Duhamel's formulation are Duhamel iterations which are one of the techniques, alongside Taylor expansions, that we use to increase the order of the error terms in our schemes. Duhamel iterations can be represented conveniently as tree diagrams which also serve as a tool for describing higher order schemes and generalising the proofs to broader classes of equations. Definitions related to Duhamel's formulation can be found in subsection \ref{Duhamel_formulation}.
	\subsection{Main examples and choices for the noise}\label{dispersive_spde_general}
	
	The main examples of \eqref{dis} are the cubic nonlinear Schr\"{o}dinger which can be written in the following form in $ \mbbT^3 $
	\begin{equation} \label{schrodinger_equation}
		i\pa_t u(t,x) + \lap u(t,x) + \lam u(t,x) \abs{u(t,x)}^2  = \Phi\xi(x,t) \sigma(u(t,x)))
	\end{equation}
	where $ \sigma(u) \in\set{\one, u(t,x)} $.
	\begin{itemize}
		\item The unknown $ u \in L^2(\RR) $ is a random process on a probability space $ \rbr{\Omega, \cF, \mbbP} $. 
		\item The case $ \sigma=1 $ corresponds to additive noise while $ \sigma = u $ corresponds to multiplicative noise.
		\item $\xi$ is a space time white noise, having the properties
		\begin{equation}
			\mbbE[\xi(t,x)]=0,\quad\mbbE[\xi(t,x)\xi(s,y)] = \delta(t-s)\delta(x-y)
		\end{equation}
		\item The smoothing operator $ \Phi:L^2(\RR^d) \rightarrow  H^s(\RR^d) $  is a Hilbert-Schmidt operator. The role of $ \Phi $ is to spatially smooth the space-time white noise $ \xi(t,x) $. This will be discussed in detail below.
		\item The constant $ \lam\in\set{-1,1} $ determines whether the equation is focusing ($ \lam=1 $) or defocusing ($ \lam=-1 $). 
	\end{itemize}
	
	The space-time Weiner process $ W(t,x) $ can be defined on the torus $ \mbbT^3 $ as 
	\begin{equation}
		W(t,x) = \sum_{n \in \mbbZ^d}W_n(t)e^{ i nx}.
	\end{equation}
	where  $ (W_n)_{n\in\mathbb{Z}^d} $ is a family of complex valued Brownian motions such that $ \text{Re}(W_n) $ and $ \text{Im}(W_n) $ are independent real value Brownian motions.
	Then, the map $ \Phi $ acts in the following way:  
	\begin{equation*}
		\Phi W(t,x) = \sum_{n \in \mbbZ^d} W_n (t) \Phi(e^{ i nx})
		=  \sum_{n \in \mbbZ^d} (\Phi W(t,x))_n e^{ i nx}.
	\end{equation*}
	In the rest of the paper without loss of generality, we will consider a map $ \Phi $ such that
	\begin{equs}
		\Phi( e^{ i nx}) = \Phi_n e^{ i nx}
	\end{equs}
	where $ \Phi_n $ are coefficients depending on $ n $. Also suppose $ \Phi^* $ is the adjoint of $ \Phi $ then we require
	\begin{equs}
		\mathrm{Tr}(\Phi\Phi^*) = \sum_{k \in \mathbb{Z}^d} \Phi_k\Phi^*_k < + \infty.
	\end{equs}
	\begin{remark}
		By expanding the Schr\"{o}dinger semigroup $ e^{it\lap} $ we will obtain terms like 
		\begin{align*}
			\lap\Phi W(t,x) &= \sum_{k \in \mathbb{Z}^d}W_k(t) k^2 \Phi_{k} e^{ i kx}.
		\end{align*}
		In that case, we will ask that $ 	\mathrm{Tr}((\Delta \Phi)^2) = \sum_{k \in \mathbb{Z}^d} \Phi_k^2 k^4 < + \infty. $
	\end{remark}
	To discretise the cylindrical Brownian motion, first define $ \Phi\chi\sim N(0,I_d) $ such that
	\begin{align}\label{eq:discrete_W}
		\int_0^\tau\Phi d W(s) &= \Phi (W(s) -W(0))
		= \sqrt{t} \frac{\Phi (W(s) -W(0))}{\sqrt{t}} =  \sqrt{t}\Phi\chi.
	\end{align} 
	If we consider the increment $ W(t_{\ell}) - W(t_{\ell-1}) $ such that $ t_{\ell} - t_{\ell-1}  = t$, we set the same definition with the notation $ \Phi\chi_{\ell} $. This increment is distributed according to the $d-$dimensional Gaussian distribution with mean $ 0 $ and covariance $ t I_d$. For double Itô integrals,  then we have, by Itô's formula, that
	\begin{align}\label{eq:discrete_W^2}
		\int_{0}^{t}\int_{0}^{s}\Phi dW(r)\Phi dW(s) =\frac{t}{2} ((\Phi\chi)^2 - \mathrm{Tr}(\Phi^2)).
	\end{align}
	We will also use the triple integral which again by Itô's formula, denoting by $ D\chi $ the $ d\times d $ diagonal matrix with $ d- $dimensional Gaussian along the diagonal,  
	\begin{equation}
		\int_{0}^{t}\int_{0}^{s}\int_0^{s_1}\Phi dW(r)\Phi dW(s_1)\Phi dW(s) = \tau^{3/2}\rbr{\frac16\rbr{\Phi\chi}^3 - \frac12\text{Tr}(\Phi D\chi\Phi^*)}
	\end{equation}
	The norm on $ \Phi\chi $ will be the standard norm on complex vector spaces. The definition of the noise and the Itô formula are given in \cite{MR3236753,MR2200233}. 
	
	The second type of examples that we will consider is the Manakov system given in Stratanovich form by
	\begin{equs}
		i \partial_t u(t,x) + \Delta u(t,x) + \lam u(t,x) \abs{u(t,x)}^2  + i\sqrt{\gamma}\sum_{j=1}^3\sigma_k\grad u(t,x) \circ dW_j(t) = 0,
	\end{equs}
	where the $ W_j(t) $ are three independent Brownian motions. Here the noise depends only on time not on space. Therefore, we do not consider an operator $ \Phi $ for smoothing the noise in space. This equation can be converted to Itô form as follows
	\begin{equs}
		i \partial_t u(t,x) + \Delta u(t,x) + \lam u(t,x) \abs{u(t,x)}^2  + i\sqrt{\gamma}\sum_{j=1}^3\sigma_k\grad u(t,x)  dW_j(t) = 0
	\end{equs}
	for $ C_\gamma = i+\frac{3\gamma}{2} $. 
	
	\subsection{Duhamel's formulation}\label{Duhamel_formulation}
	
	One can rewrite the equation \eqref{dis} using Duhamel's formulation, that is, 
	\begin{equation}\label{dis}
		\begin{aligned}
			u(t) & =
			e^{i t \mathcal{L}} v  -  i e^{i t \mathcal{L}  }  \int_0^t  e^{-i s \mathcal{L}  }  \left(\vert \nabla\vert^\alpha p\left(u(s), \overline u(s)\right) \right) ds  \\ & - i e^{i t \mathcal{L}  }  \int_0^t  e^{-i s \mathcal{L}  } \left(  \vert \nabla\vert^\beta  f(u(s,x),\overline u(s,x)) d W(s,x) \right).
		\end{aligned}
	\end{equation}
	In the case of cubic NLS with multiplicative noise, one has
	\begin{equs} \label{NLS_duhamel} \begin{aligned}
			u(t) &= e^{i t \Delta}v - ie^{i t \Delta}\int_0^t e^{-i s \Delta}u(s)\abs{u(s)}^2ds \\ &- ie^{i t \Delta}\int_0^t e^{-i  s \Delta} u(s) \Phi dW(s).
		\end{aligned}
	\end{equs}
	The last term in the previous expression is called the \textit{stochastic convolution}, in which we often omit the dependency in space of $ W(s) $ for notational brevity.
	One basic idea for building an approximation of the solution is to iterate this formulation.
	The first order iteration substitutes the integral form of the solutions $ u $ with the twisted variable $ e^{i t \lap}v $ to give
	\begin{equs}
		u(t) &= e^{i t \lap}v - ie^{i t \lap}\int_0^t e^{-is\lap}\rbr{e^{i s \lap}v}^2e^{-i s \lap}\bar{v}ds \\
		&- ie^{it\lap}\int_0^t e^{-i s\lap } \rbr{e^{i s \lap}v}\Phi dW(s) + R(t)
	\end{equs}
	where $ R(t) $ is  a remainder that is not of order $ \mathcal{O}(t^{\frac{3}{2}}) $ as some terms of order $ \mathcal{O}(t) $ are missing.
	The second order iteration will include all terms that scale linearly in time. For instance if we  iterate the stochastic convolution we obtain
	\begin{equs}
		\label{stochastic_iteration}
		e^{i t \Delta}\int_0^t e^{-is\Delta}\rbr{e^{i s\Delta  }\int_0^ s e^{-is_1\Delta}\rbr{e^{i s_1\Delta} v}\Phi dW( s_1) } \Phi dW( s)\sim t,
	\end{equs}
	which is of order $ \mathcal{O}(t) $ because $ dW(t) $ scales as $ \sqrt{t} $.
	\begin{remark}
		Our numerical schemes will be based upon truncations of the iterated mild form. To perform numerical approximations one is required to expand the operator $ e^{i s \lap} $ and thus we need regularity assumptions on the initial data, that is we need $ v \in H^{s+\alpha} $ where $ s > \max\{1,d/2\} $ and $ \alpha $ depends on the level of the Taylor expansion in the scheme and is the result of expanding the differential operators. The approach presented in this paper is intended to reduce the impact of Taylor expansions on the regularity of the scheme, reducing the value of $ \alpha $. In the sequel we will always assume additional regularity $ s $ in the schemes.
	\end{remark} 
	In order to improve significantly the regularity on the initial data, one wants to explode the resonances in these oscillatory integrals.
	For that, we start by rewriting the Duhamel's formulation in Fourier space:
	\begin{equs}
		u_k(t) & = e^{-i t k^2}v_k - \sum_{k=-k_1 + k_2 + k_3} ie^{-i t k^2}\int_0^t e^{i s k^2} \bar{u}_{k_1}(s) u_{k_2}(s) u_{k_3}(s) ds \\ &  - \sum_{k=k_1 +k_2} ie^{-i t k^2}\int_0^t e^{i  s k^2} u_{k_1}(s) \Phi_{k_2} dW_{k_2}(s)
	\end{equs}
	where $ u_k $ is the $k$th Fourier coefficient of the solution of \eqref{NLS_duhamel}. We have used the fact that $ e^{i t \Delta} $ is mapped in Fourier space to  $  e^{-i t k^2}$. Moreover, pointwise products in physical space are sent to convolution products in Fourier space. Then iterating this formulation, one gets the same iterated integrals but this time written in Fourier space. For example, \eqref{stochastic_iteration} becomes:
	\begin{equs}
		&	\sum_{k = k_1 + k_2 +k_3}	e^{-i t k^2}\int_0^t e^{isk^2} \\ 
		& \rbr{e^{-i s (k_1 + k_2)^2  }\int_0^ s e^{is_1 (k_1+k_2)^2}\rbr{e^{-i s_1 k_1^2} v_{k_1}}\Phi_{k_2} dW_{k_2}( s_1) } \Phi_{k_3} dW_{k_3}( s).
	\end{equs}
	
	\subsection{Decorated trees}
	For encoding the previous integrals, one can use an extended version of the decorated trees formalism introduced in \cite{B.S2022}. Let  $\Lab$ a finite set  and  frequencies $ k_1,...,k_m \in \Z^{d}$. 
	We suppose that $ \mathfrak{L} $  factorises into $ \mathfrak{L} = \mathfrak{L}_{+} \cup \mathfrak{L}_- $, 
	where $ \mathfrak{L}_+ $ parametrises a family of polynomials $ (P_{\Labhom})_{\Labhom \in \Lab_+} $ in the frequencies that represent operators with constant coefficients in Fourier space. 
	The other set $ \Lab_- $ encodes noises $ (\xi_{\mft })_{\mft \in \Lab_-} $ that we will use for the stochastic integration.
	\begin{definition}		\label{decorated_trees}
		We define the set of decorated trees $\mcT$ as elements of the form $T_\mfe^{\mfn,\mff}$ where 	
		\begin{itemize}	
			\item $ T $ is a non-planar rooted tree with root node $ \varrho $, edge set $ E_T $ and node set $ N_T $. We consider only planted trees which means that $ T $ has only one edge connecting its root $ \varrho $.		
			\item $ \mfe:E_T\rightarrow \mfL\times\set{0,1} $ is the set of edge decorations $ \mfe(\cdot) = (\mft(\cdot), \mfp(\cdot)) $ where the first component selects the correct polynomial $ P_{\mft} $ when $ \mft \in \Lab_+ $ or the correct noise $ \xi_{\mathfrak{l}} $ when $ \mathfrak{l} \in \Lab_- $. Edges decorated by elements of $ \Lab_- $ are terminal edges.	The second decoration encodes conjugation when it is equal to $ 1 $.
			\item $ \mfn:N_T\nin\set{\varrho }\rightarrow\mbbN $ are node decorations on the inner nodes of the trees encoding the indices for the time monomials $ s \mapsto  s^{\mfn(\cdot)}$.		
			\item $ \mff:N_T\nin\set{\varrho}\rightarrow \mathbb{Z}^d $ are node decorations  satisfying the relation for every inner node $ u $
			\begin{equs} \label{frequencies_identity}
				(-1)^{\mfp(e_u)}\mff(u)=\sum_{e = (u,v)\in E_T}(-1)^{\mfp(e)}\mff(v)
			\end{equs}
			where $ e_u$ is the edge outgoing $u$ of the form $ (w,u) $. From this definition, one
			can see that the node decorations  $(\mff(u))_{u \in L_T}$ determine the decorations of the inner nodes. We assume
			that the node decorations at the leaves are linear combinations of the $k_i$ with coefficients in $ \lbrace -1,0,1 \rbrace$.
		\end{itemize}
	\end{definition} 
	
	In the sequel, we will disregard some decorations on a tree $ T_\mfe^{\mfn,\mff} $: If we remove the node decorations, we will denote the tree as  $ T_\mfe $ and if it is only $ \mfn $, we denote the tree as $ T_\mfe^{\mff} $ (the corresponding set is $ \CT_0 $).
	We denote by $ H $  the (unordered) forests composed of trees in $ \CT $ (including the empty forest denoted by $ \one $). Their linear spans are denoted by $\CH$. Equipped with the forest product denoted by $ \cdot $ (union of two forests), $\CH$ is a commutative algebra.

	Before presenting the iterated integrals associated to these trees, we introduce a symbolic notation. An  edge decorated by $ o = (\mft, \mfp) $ with $ \mft \in \Lab_+ $ is denoted by $ \mathcal{I}_{o} $ otherwise it is denoted by $ \Xi_o $. The symbol $ \mathcal{I}_{o} (\lambda^{\ell}_{k}
	\cdot) : \CH \rightarrow \CH $ is viewed as the operation that merges all the roots of the trees composing the forest into one node decorated by $ (\ell,k) \in \mathbb{N} \times \mathbb{Z}^d$. We obtain
	a decorated tree which is then grafted onto a new root with no decoration. In the sequel, we will use the following shorthand notation: $ \lambda_k $ instead of $ \lambda_k^0 $. If the condition \eqref{frequencies_identity} is not
	satisfied on the argument, then $ \mathcal{I}_{o} (\lambda^{\ell}_{k}
	\cdot) $  gives zero. The edges of the type $ \Xi_{o} $ are always terminal edges and they are of the form $ \Xi_{o}(\lambda_k) $.
	Below, we illustrate the symbolic notation on various examples. We suppose that $ \mfL_+ = \set{\mft_1,\mft_2}$, $ \mfL_- = \set{\mfl }$ and  $P_{\mathfrak{t}_1}(\lambda) = -\lambda^2$ and $P_{\mathfrak{t}_2}(\lambda) = \lambda^2$. Then, one has
	\begin{equs}			
		T_1 = \begin{tikzpicture}[node distance = 14pt, baseline=7]				
			\node[inner] (root) at (0,0) {};			
			\node[inner] (inner4) [above of = root] {};			
			\node[leaf] (leaf1) [above left of = inner4] {$ k_1 $};		
			\node[leaf] (leaf2) [above of = inner4] {$ k_2 $};		
			\node[leaf] (leaf3) [above right of = inner4] {$ k_3 $};		
			\draw[int] (root.north) -- (inner4.south);
			\draw[positive] (inner4.north) -- (leaf2.south);
			\draw[conj] (inner4.north west) -- (leaf1.south east);
			\draw[positive] (inner4.north east) -- (leaf3.south west);		  
		\end{tikzpicture} =	  \CI_{(\mft_2,0)}\rbr{\lam_{k}\CI_{(\mft_1,1)}\rbr{\lam_{k_1}}\CI_{(\mft_1,0)}\rbr{\lam_{k_2}}\CI_{(\mft_1,0)}\rbr{\lam_{k_3}}} 
	\end{equs}
	where $ k = -k_1 + k_2 + k_3 $.
	The frequency decorations appear on the leaves of the previous tree. One does not have to make explicit the frequency decoration  for the inner nodes as they are determined by the ones coming from the leaves.  We have also omitted the forest product $ \cdot $ as $ \CI_{(\mft_1,1)}\rbr{\lam_{k_1}}\CI_{(\mft_1,0)}\rbr{\lam_{k_2}} $ is a shorthand notation for $ \CI_{(\mft_1,1)}\rbr{\lam_{k_1}} \cdot \CI_{(\mft_1,0)}\rbr{\lam_{k_2}} $. We have used the following pictorial dictionary:
	\begin{equs}
		\begin{tabular} { 
				| m{2.0em} | m{5.0em}| m{20.0em} |}
			\hline
			\textbf{Edges}  & \textbf{Notation} & \textbf{Operator} \\
			\hline
			\begin{tikzpicture}[node distance = 18pt, baseline=7]				
				\node[] (root) at (0,0) {};			
				\node[] (inner4) [above of = root] {};			
				\draw[positive] (root.north) -- (inner4.south);	  
			\end{tikzpicture} & $ \CI_{(\mathfrak{t}_1,0)}(\lambda_k \cdot)$ &  $e^{i t P_{\mathfrak{t_1}}(k)} =  e^{-i t k^2}$ \\
			\hline
			\begin{tikzpicture}[node distance = 18pt, baseline=7]				
				\node[] (root) at (0,0) {};			
				\node[] (inner4) [above of = root] {};			
				\draw[conj] (root.north) -- (inner4.south);	  
			\end{tikzpicture} & $\CI_{(\mathfrak{t}_1,1)}(\lambda_k \cdot)$  & $e^{-i t P_{\mathfrak{t_1}}(k)} = e^{i t k^2}$  \\
			\hline
			\begin{tikzpicture}[node distance = 18pt, baseline=7]				
				\node[] (root) at (0,0) {};			
				\node[] (inner4) [above of = root] {};			
				\draw[int] (root.north) -- (inner4.south);	  
			\end{tikzpicture}   & $\CI_{(\mathfrak{t}_2,0)}(\lambda_k \cdot)$  & $-i\int^{t}_0 e^{i s P_{\mathfrak{t}_2}(k)} \cdots d s = -i\int^{t}_0 e^{i s k^2} \cdots d s$  \\
			\hline
			\begin{tikzpicture}[node distance = 18pt, baseline=7]				
				\node[] (root) at (0,0) {};			
				\node[] (inner4) [above of = root] {};			
				\draw[conj-int] (root.north) -- (inner4.south);	  
			\end{tikzpicture}  & $\CI_{(\mathfrak{t}_2,1)}(\lambda_k \cdot)$  & $-i \int^{t}_0 e^{-i s P_{\mathfrak{t}_2}(k)} \cdots d s=  -i \int^{t}_0 e^{-i s k^2} \cdots d s$  \\
			\hline
		\end{tabular}
	\end{equs}
	The previous dictionary is valid when $ \CI_{o}(\lambda_k F) $ is such that the forest $ F $ does not contain $ \Xi_o $ among its trees.
	The decorated tree $ T_1 $ is an abstract version of the following integral:
	\begin{equs}
		- i e^{-i t k^{2}} \int^{t}_0 e^{i s k^2} e^{i s k_1^2}  e^{-i s k_2^2}  e^{-i s k_3^2}d s.
	\end{equs}
	Below, we give another example with their associated iterated integrals when one considers noises $ \Xi_{o}  $:
	\begin{equs}	
		T_2 =    \begin{tikzpicture}[node distance = 14pt, baseline=7]	
			\node[inner] (root) at (0,0) {};	
			\node[inner] (inner4) [above of = root] {};	
			\node[leaf] (leaf1) [above left of = inner4] {$ k_1 $};	
			\node[leaf] (leaf2) [above right of = inner4] {$ k_2 $};	
			\draw[int] (root.north) -- (inner4.south);	
			\draw[Phi] (inner4.north east) -- (leaf2.south west);	
			\draw[positive] (inner4.north west) -- (leaf1.south east);	 		
		\end{tikzpicture}  = \CI_{(\mft_2,0)}\rbr{\lam_{k} \CI_{(\mft_1,0)}\rbr{\lam_{k_1}} \Xi_{(\mfl,0)}(\lambda_{k_2})}    \equiv \int_{0}^{t}e^{is k^{2} }\left(e^{-i s k_{1}^{2} }\right)\Phi_{k_{2}}dW_{k_{2}}(s)
	\end{equs}
	where $ k = k_1 + k_2 $ and we have used a zigzag brown line for encoding the noise $ \Xi_{(\mfl,0)} $. The third example corresponds to the iteration of the stochastic integral within itself: 
	\begin{equs}
		T_3 & =  \begin{tikzpicture}[node distance = 14pt, baseline=14]
			\node[inner] (root) at (0,0) {};
			\node[inner] (inner1) [above of = root] {};
			\node[inner] (inner2) [above left of = inner1] {};
			\node[inner] (inner3) [above of = inner2] {};
			\node[leaf] (leaf1) [above left of = inner3] {$ k_1 $};
			\node[leaf] (leaf2) [above right of = inner3] {$ k_2 $};
			\node[leaf] (leaf3) [above right of = inner1] {$ k_3 $};
			\draw[int] (root.north) -- (inner1.south);
			\draw[positive] (inner1.north west) -- (inner2.south east);
			\draw[int] (inner2.north) -- (inner3.south);
			\draw[positive] (inner3.north west) -- (leaf1.south east);
			\draw[Phi] (inner3.north east) -- (leaf2.south west);
			\draw[Phi] (inner1.north east) -- (leaf3.south west); 
		\end{tikzpicture}  =  \CI_{(\mft_2,0)}\rbr{\lam_{k} \CI_{(\mft_1,0)}\rbr{\lam_{k_{12} } T_2} \Xi_{(\mfl,0)}(\lambda_{k_3})}
		\\
		& \equiv \int_0^t e^{ik^2s}\left(e^{-i s k_{12}^2}\int_0^s e^{i s_1 k_{12}^2}e^{-ik_1^2s_1}\Phi_{k_2}dW_{k_2}(s_1)\right)\Phi_{k_3}dW_{k_3}(s)
	\end{equs}
	where $ k = k_1 + k_2 + k_3 $ and $ k_{12} = k_1 + k_2 $.
	Our last example iterates the stochastic integral inside the deterministic one,
	\begin{equs}
		T_4 & = \begin{tikzpicture}[node distance = 14pt, baseline=7]				
			\node[inner] (root) at (0,0) {};			
			\node[inner] (inner4) [above of = root] {};		
			\node[inner] (inner5) [above of = inner4] {};		
			\node[inner] (inner6) [above of = inner5] {};		
			\node[leaf] (leaf1) [above left of = inner4] {$ k_3 $};		
			\node[leaf] (leaf2) [above right of = inner6] {$ k_2 $};		
			\node[leaf] (leaf4) [above left of = inner6] {$ k_1 $};
			\node[leaf] (leaf3) [above right of = inner4] {$ k_4 $};		
			\draw[int] (root.north) -- (inner4.south);
			\draw[positive] (inner4.north) -- (inner5.south);
			\draw[int] (inner5.north) -- (inner6.south);
			\draw[conj] (inner4.north west) -- (leaf1.south east);
			\draw[positive] (inner4.north east) -- (leaf3.south west);	
			\draw[Phi] (inner6.north east) -- (leaf2.south west);	
			\draw[positive] (inner6.north west) -- (leaf4.south east);			  
		\end{tikzpicture}  = \CI_{(\mft_2,0)}\rbr{\lam_{k}\CI_{(\mft_1,1)}\rbr{\lam_{k_3}}\CI_{(\mft_1,0)}\rbr{\lam_{k_{12}} T_2}\CI_{(\mft_1,0)}\rbr{\lam_{k_4}}} 
		\\ &  = 	- i e^{-i t k^{2}} \int^{t}_0 e^{i s k^2} e^{i s k_3^2}  e^{-i s k_{12}^2} \left( \int_{0}^{s}e^{i s_1 k_{12}^{2}}\left(e^{-i s_1 k_{1}^{2} }\right)\Phi_{k_{2}}dW_{k_{2}}(s_1) \right)  e^{-i s k_4^2}d s,
	\end{equs}
	where this time $ k = -k_1 + k_2 + k_3 + k_4 $. One can define the iterated integrals recursively based on the inductive construction of the decorated trees via a map $ \Pi :  \CH \rightarrow \mathcal{C} $. The space $  \mathcal{C} $ contains  functions of the form $ z \mapsto \sum_j Q_j(z) e^{i z P_j(k_1,...,k_n)} $  and the $ Q_j(z) $ are polynomials in
	$ z $, the $ P_j $ are polynomials in  $ k_1,...,k_n $, and each $ Q_j $ depends implicitly on $k_i$. The space $ \mathcal{C} $ contains also iterated integrals in the $ W_{k_i} $ with integrands depending on the the $ k_i $. 
	The map $ \Pi $ is defined for every forest $ \bar{F} $, and $ F $ not containing noises $ \Xi_o $, by
	\begin{equs}
		\label{def_Pi}
		\begin{aligned}
			\Pi (\mathcal{I}_{o_1}(\lambda^{\ell}_{k} \bar{F}))(t) &= e^{i t  P_{o_1}(k)} t^{\ell} (\Pi \bar{F})(t) \\
			\Pi (\mathcal{I}_{o_2}(\lambda^{\ell}_{k} F))(t) &= - i |\nabla|^\alpha(k)\int_0^t e^{i s P_{o_2}(k)} s^{\ell} \Pi(F)(s)d s \\
			\Pi (\mathcal{I}_{o_2}(\lambda^{\ell}_{k} F \Xi_o(\lambda_{\bar{k}})))(t) &= - i |\nabla|^\beta(k)\int_0^t e^{i s P_{o_2}(k)} s^{\ell}(\Pi F)(s) \Phi_{\bar{k}}d W_{\bar{k}}( s). 
		\end{aligned}
	\end{equs} 
	The iterated integrals come from Duhamel's formula in Fourier space, see \eqref{dis}.
	We have supposed that $ \Lab_+ $ split into two subsets $ \Lab_{1,+} $ and $  \Lab_{2,+}  $. 
	The set $  \Lab_{2,+}  $ corresponds to integrals in time and $ \Lab_{1,+} $ is just the propagator. We use the notation $ o_i =(\mathfrak{t}_i,\mathfrak{p}_i) $ for saying that $ \mathfrak{t}_i \in \Lab_{i,+} $ with $ i \in \lbrace 1,2 \rbrace $.
	We extend multiplicatively the definition of $ \Pi $ for the forest product and we set $ (\Pi\one)(t) = 1 $.
	
	\subsection{Approximation with a truncated tree series}
	
	Before showing that the previous iterated integrals can actually approximate the solution of \eqref{dis}. We compute their order in $ t $. Given a decorated tree  $ T_\mfe^{\mfn,\mff} $, we define the order of a tree denoted by $ \vert \cdot \vert_{\text{ord}} $
	by
	\begin{equation*}
		\vert  T_\mfe^{\mfn,\mff} \vert_{\text{ord}} = \sum_{e \in E_T }  \left( \one_{\lbrace \mathfrak{t}(e) \in \Lab_{2,+} \rbrace} + \frac{1}{2} \one_{\lbrace \mathfrak{t}(e) \in \Lab_{-} \rbrace} \right) + \sum_{v \in N_T \setminus \lbrace  \varrho  \rbrace} \mfn(v).
	\end{equation*} 
	In the definition of the order, the integrals in time give a contribution $  1 $ whereas the stochastic integrals give a contribution $ \frac{1}{2} $ which is due to the scaling of the Brownian motion. Then, one has 
	\begin{equation*}
		\left( \Pi  T_\mfe^{\mfn,\mff} \right)(t) = \mathcal{O}(t^{\vert  T_\mfe^{\mfn,\mff} \vert_{\text{ord}}}).
	\end{equation*}
	For describing the solution $ u_k $ in the Duhamel's formula in Fourier space, one needs to select the correct decorated trees. One way to do it is to describe them as decorated trees generated by rules coming from the nonlinearity of the equation.
	
	A rule is then a map $R$ assigning to each element of $ \Lab \times \lbrace 0,1 \rbrace $ a nonempty collection of tuples in $ \Lab \times \lbrace 0,1 \rbrace $. The relevant rule describing the equation \eqref{dis} when $ p(u,\bar{u}) = u^N \bar{u}^M $ and  $f(u,\overline u) = u $ is given by 
	\begin{equs}
		R((\mathfrak{l},p)) & = \lbrace ()\rbrace, \quad 
		R((\mathfrak{t}_1,p))  = \lbrace (), ( (\mathfrak{t}_2, p)) \rbrace
		\\
		R((\mathfrak{t}_2,p)) & = 	\lbrace ((\mathfrak{t}_1,p)^N, (\mathfrak{t}_1, p+1)^M ), ((\mathfrak{t}_1,p), (\mathfrak{l},p)) \rbrace,
	\end{equs}
	where $ (\mathfrak{t}_1,p)^N $ means that $ (\mathfrak{t}_1,p) $ is repeated $ N $ times.
	We recall the definition of a decorated tree generated by a rule $ R $.

	\begin{definition}
		A decorated tree $ T_\mfe^\mff\in \mcT_0 $ is generated by the rule $ R $ if for every node $ u\in N_T $ one has 
		\begin{equation}
			\cup_{e\in E_u}\rbr{\mft(e),\mfp(e)}\in R(e_u)
		\end{equation}
		where $ E_u \subset E_T$ are the edges of the form $ (u,v) $ and $ e_u $ is the edge of the form $ w,u $. The set of decorated trees generated by $ R $ is denoted by $ \mcT_0(R) $.
	\end{definition}
	
	Alternatively, one can describe the set $ \mcT_0(R) $ inductively by:
	\begin{equs}
		& \CT_0(R)  = \lbrace \CI_{(\Labhom_1,0)}( \lambda_k \CI_{(\Labhom_2,0)}( \lambda_k   \prod_{i=1}^N T_i \prod_{j=1}^M  \tilde T_j  ) ), \CI_{(\Labhom_1,0)}(\lambda_k) \,  \\ & \quad \CI_{(\Labhom_1,0)}( \lambda_k \CI_{(\Labhom_2,0)}( \lambda_k T \, \Xi_{(\mathfrak{l},0)}(\lambda_{\bar{k}}) ) ): \,T_i, T \in \CT_0(R), \, \tilde{T}_j \in \bar{\CT}_0(R), \, k, \bar{k} \in \Z^{d}  \rbrace \\
		& \bar{\CT}_0(R)  = \lbrace \CI_{(\Labhom_1,1)}( \lambda_k \CI_{(\Labhom_2,1)}( \lambda_k   \prod_{i=1}^N T_i \prod_{j=1}^M  \tilde T_j  ) ), \CI_{(\Labhom_1,1)}(\lambda_k) \,  \\ & \quad \CI_{(\Labhom_1,1)}( \lambda_k \CI_{(\Labhom_2,1)}( \lambda_k T \, \Xi_{(\mathfrak{l},1)}(\lambda_{\bar{k}}) ) ): \,T_i, T \in \bar{\CT}_0(R), \, \tilde{T}_j \in \CT_0(R), \, k, \bar{k} \in \Z^{d}  \rbrace
	\end{equs}
	We define $   \CT_0^k(R)  $ as the subspace of $ \CT_0(R) $ such that the frequency decoration on the nodes connected to the root is $ k $. For $ r\in\mathbb{N} $ we set
	\begin{equation}
		\mcT^{r,k}_0\rbr{R}=\set{T_\mfe^\mff\in \mcT^k_0\rbr{R},|T^\mff_\mfe|_{\text{ord}}\leq r}
	\end{equation}
	which corresponds to the trees of order in $ t $ less than $ r $.
	We will now introduce the elementary differential, which is required to present the expansion of the solution in terms of iterated integrals. We start by defining the symmetry factor of the tree $ T^{\mfn,\mff}_{\mfe}\in\mcT $ by considering only the edge decoration, i.e, setting $ T^{\mfn,\mff}_\mfe := T_\mfe $ and then setting $ S(\boldsymbol{1}) = 1 $ and working inductively defining 
	\begin{equation}
		S(T):=\rbr{\prod_{i,j}S(T_{i,j})^{\gamma_{i,j}}\gamma_{i,j}!}
	\end{equation}
	for the tree $ \prod_{i,j}\CI_{(\mft_{t_i}, p_i)}\rbr{T_{i,j}}^{\gamma_{i,j}} $ where $T_{i,j} \neq T_{i,\ell}$ for $j \neq \ell$.
	We now define the elementary differentials denoted by $ \Upsilon(T)(v) $ for 
	\begin{equs}
		T  & = 
		\CI_{(\Labhom_1,a)}\left( \lambda_k \CI_{(\Labhom_2,a)}( \lambda_k   \prod_{i=1}^n \CI_{(\Labhom_1,0)}( \lambda_{k_i} T_i) \prod_{j=1}^m \CI_{(\Labhom_1,1)}( \lambda_{\tilde k_j} \tilde T_j)  ) \right), \quad a \in \lbrace 0,1 \rbrace  
		\\ 	\bar{T}  & = 
		\CI_{(\Labhom_1,a)}\left( \lambda_k \CI_{(\Labhom_2,a)}( \lambda_k   \prod_{i=1}^n \CI_{(\Labhom_1,0)}( \lambda_{k_i} T_i) \prod_{j=1}^m \CI_{(\Labhom_1,1)}( \lambda_{\tilde k_j} \tilde T_j)  \Xi_{(\mathfrak{l},a)}(\lambda_{\bar{k}})) \right), 
	\end{equs}
	by
	\begin{equs}
		\Upsilon(T)(v) \, & { :=}  \partial_v^{n} \partial_{\bar v}^{m} p_a(v,\bar v) \prod_{i=1}^n  \Upsilon( \CI_{(\Labhom_1,0)}\left( \lambda_{k_i}  T_i \right) )(v)  \prod_{j=1}^m \Upsilon( \CI_{(\Labhom_1,1)}( \lambda_{\tilde k_j}\tilde T_j ) )(v)
		\\ \Upsilon(T)(v) \, & { :=}  \partial_v^{n} \partial_{\bar v}^{m} f_a(v,\bar v) \prod_{i=1}^n  \Upsilon( \CI_{(\Labhom_1,0)}\left( \lambda_{k_i}  T_i \right) )(v)  \prod_{j=1}^m \Upsilon( \CI_{(\Labhom_1,1)}( \lambda_{\tilde k_j}\tilde T_j ) )(v),
	\end{equs}
	where 
	\begin{equs} 
		\label{mid_point_rule}
		\begin{aligned}
			\Upsilon(\CI_{(\Labhom_1,0)}( \lambda_{k})  )(v)  \,  { :=}  v_k, \quad 
			\Upsilon(\CI_{(\Labhom_1,1)}( \lambda_{k})  )(v)  \, & { :=}  \bar{v}_k.
		\end{aligned}
	\end{equs}
	Above, we have used the notation:
	\begin{equs}
		p_0(v,\bar{v}) = p(v,\bar{v}), \quad  f_0(v,\bar{v}) = f(v,\bar{v}),\quad p_{1}(v,\bar{v}) = \overline{p(v, \bar{v})}, \quad f_{1}(v,\bar{v}) = \overline{f(v, \bar{v})}. 
	\end{equs}
	\begin{proposition} \label{tree_series}
		The tree series given by 
		\begin{equs}\label{genscheme}
			U_{k}^{r}(t, v) =   \sum_{T \in \CT^{r,k}_{0}(R)} \frac{\Upsilon( T)(v)}{S(T)} (\Pi   T )(t)
		\end{equs}
		is the $k$th Fourier coefficient of a solution of $ \eqref{dis} $ up to order $ r $ which means that one has
		\begin{equs}
			u_{k}(t) - 	U_{k}^{r}(t, v) = \mathcal{O}(t^{r+\frac{1}{2}}).
		\end{equs}
	\end{proposition}
	\begin{proof}
		The proof follows the same lines as the one given in \cite[Prop. 4.3]{B.S2022}. The only difference is the fact that stochastic integrals scale differently as $ t^{\frac{1}{2}} $.
	\end{proof}
	
	One wants to discretise the expansion $ U^r_k(t,v) $ by discretising each $ (\Pi T)(t) $ while accounting for the resonance within the discretisation. The idea is to minimise the regularity asked on the initial data $ v $ and it turns out that it is possible to do this for lower order schemes but not for high order ones. This is the subject of the next theorem:

	\begin{theorem} \label{main_theorem_general} 
		Assuming that the initial data $ v $ is in some $ H^n $, one can find a resonance discretisation $ U^{n,r}_k(t,v) $ in the flavour of \cite{B.S2022} such that 
		\[
		U^{n,r}_k(t,v) - u_k(t) =\begin{cases}\sum_{T\in\mcT_0^{r,k}}\mcO\rbr{t^{r+1/2}\msL^r_{\text{\tiny{low}}}\rbr{T,n}\Ups\rbr{T}(v)}\quad r \leq 1 \\ 
			\sum_{T\in\mcT_0^{r,k}}\mcO\rbr{t^{r+1/2}\msL^{r-1}\Ups\rbr{T}(v)}\quad r > 1
		\end{cases} \]
		where $ \msL^r_{\text{\tiny{low}}}(T,n) $ is given as in \cite[Def. 3.11]{B.S2022}. 
	\end{theorem}
	
	\begin{remark}
		One can observe that the class of schemes introduced in Theorem~\ref{main_theorem_general} requires the same regularity as classical schemes for an order $ r > 1 $. This is due to particular iterated integrals for which one cannot evaluate the resonance integral explicitly due to the presence of a stochastic term. We will describe this obstruction in section~\ref{Sec::3} and provide an iterated integral in Proposition~\ref{counter_example} with the given local error. However, resonance schemes \textit{are} effective at orders lower than $ r\leq1 $. We present several examples demonstrating this result, such as the NLS with multiplicative noise and the Manakov system. We refrain from recalling the definition of $  \msL^r_{\text{\tiny{low}}}(T,n) $, as the number of iterated integrals is small when $ r\leq 1 $ and the proof will be similar to that for the stochastic NLS with multiplicative noise.
	\end{remark}

	\section{Deriving low regularity schemes}
	\label{Sec::3}
	In this section, we describe the derivation of the schemes in general using Schr\"odinger's equation as an example. We show that by iterating the mild form of the equation and performing resonance analysis, we can lower the order of the differential operator acting on the initial conditions. Subsequently by Taylor expansion we can approximate the iterated integrals to produce numerical schemes. We will demonstrate the process for the nonlinear term and the stochastic term. We will show why this process fails at higher orders using an order $ \mcO(t^{3/2}) $ mixed (stochastic and deterministic) integral, proving the high regularity stated in Theorem~\ref{main_theorem_general} when $ r > 1 $.
	
	\subsection{Main ideas of the resonance approach}
	To obtain the schemes we recall the expansion obtained from Duhamel's formulation in Fourier space at order $ \mcO(t^{3/2}) $: 
	\begin{equs}
		\label{picard32}
		\begin{aligned}
			u(t)& =e^{it\Delta}v-i e^{i t \Delta}\int_{0}^{t}e^{-is\Delta}\left(e^{-is\Delta}\bar{v}\right)\left(e^{is\Delta}v\right)\left(e^{is\Delta}v\right)d s \\
			& -i e^{i t \Delta}\int_{0}^{t}e^{-is\Delta}\left(e^{is\Delta}v\right)\Phi dW(s)\notag\\
			& - e^{i t \Delta}\int_{0}^{t}e^{-is\Delta}\left(e^{is\Delta}\int_{0}^{s}e^{-i\Delta s_{1}}\left(e^{i\Delta s_{1}}v\right)\Phi dW( s_{1})\right)dW( s)\notag \\
			& +R(t).\notag
		\end{aligned}
	\end{equs} 
	where the remainder $ R(t) $ is of order $ \mathcal{O}(t^{3/2}) $ and $ v=u(0) $ is the initial data. The numerical scheme is derived by discretising the terms in \eqref{picard32}. First we will consider the nonlinear term in \eqref{picard32}. To do this we use the Fourier transform of the initial data $ v=\sum_{k\in\mbbZ^d}e^{ikx}v_k $. Substituting this into \eqref{picard32}, the nonlinear term becomes, for $ k=-k_1+k_2+k_3 $
	\begin{equs}
		\int_{0}^{t}e^{-is\Delta}\left(e^{-is\Delta}\bar{v}\right)\left(e^{is\Delta}v\right)\left(e^{is\Delta}v\right)d s & = \sum_{k = -k_1 + k_2 + k_3}e^{ikx}\overline{v}_{k_1}v_{k_2}v_{k_3}
		\\ &	\int_{0}^{t}e^{i s P(k_1,k_2,k_3)}ds,
	\end{equs}
	where $P(k_1,k_2,k_3)=2k_1^2-2k_1(k_2+k_3)+2k_2k_3$ is the resonance structure of the iterated integral. In the sequel, we will use the shorthand notation $ P $ instead of $ P(k_1,k_2,k_3) $. We note that the powers of the $k_i$ correspond to the orders of the differential operator in space, while the subscripts link the operator to the function it acts upon. In the function $ P $ for instance $ k_i^2\mapsto\lap v $ while $ k_ik_j\mapsto (\grad v)^2 $ for $i\neq j$, and $ i,j\in\set{1,2,3} $. Hence all terms in $ P $ correspond to first order differential operators, apart from $ 2k_1^2 $, and the power of the resonance approach is our ability to write the integral in closed form, i.e. \begin{equs} \label{discretisation_T_1}
		\int_0^{t}e^{2ik_1^2 s}d s= \frac{e^{2ik_1^2 t}-1}{2ik_1^2}. 
	\end{equs}
	The remaining part of the operator, $ e^{isP} $, can be discretised by Taylor expansion of the operator $ e^{i s \mathcal{L}_{\text{low}}}=e^{is\rbr{P- 2k_1^2}} $ without the higher order differential operators acting on the initial conditions. 
	One connects this term with the tree formalism introduced before. Indeed, one has 
	\begin{equs}\label{tree_nonlinear}
		T_1 =  \begin{tikzpicture}[node distance = 14pt, baseline=7]				
			\node[inner] (root) at (0,0) {};			
			\node[inner] (inner4) [above of = root] {};			
			\node[leaf] (leaf1) [above left of = inner4] {$ k_1 $};		
			\node[leaf] (leaf2) [above of = inner4] {$ k_2 $};		
			\node[leaf] (leaf3) [above right of = inner4] {$ k_3 $};		
			\draw[int] (root.north) -- (inner4.south);
			\draw[positive] (inner4.north) -- (leaf2.south);
			\draw[conj] (inner4.north west) -- (leaf1.south east);
			\draw[positive] (inner4.north east) -- (leaf3.south west);		  
		\end{tikzpicture}, \quad	(\Pi T_1)(t)= -i \int_{0}^{t}e^{i s P}ds,
	\end{equs}
	and with a quick computation, one also realises that
	\begin{equs}
		S(T_1) = 2, \quad \Upsilon(T_1)(v) = 2\overline{v}_{k_1}v_{k_2}v_{k_2}.
	\end{equs}
	One obtains a factor $ 2 $ from the symmetry factor because we do not consider the frequency decorations in its definition. One can set from \eqref{discretisation_T_1}
	\begin{equs}
		(\Pi^{n,r} T_1 )(t) = - \frac{e^{2ik_1^2 t}-1}{2k_1^2}
	\end{equs}
	with $ n =1 $ and $ r = 1 $ coming from the fact that we need $ v \in H^1 $ for making sense of this discretisation.

	The discretisation of stochastic terms is slightly different. We perform the same operations but the noise will have it's own Fourier coefficients coupling it with the smoothing operator $ \Phi $.  For example, if we consider the first stochastic integral in \eqref{picard32} and apply the Fourier transform to the initial data and the white noise then we have 
	\begin{equation}
		\int_{0}^{t}e^{-is\Delta}\left(e^{is\Delta}v\right)\Phi dW( s)=\sum_{k = k_1 + k_2}e^{ik x}\Phi_{k_2} v_{k_1}\int_{0}^{t} e^{i s(k_2^2+2k_1k_2)} dW_{k_2}( s).
	\end{equation}
	At this point, if we proceed just as we did for the deterministic integral then we would note that, by Taylor expansion, we have 
	\begin{equation}\label{eq:stoch-taylor_exp}
		\int_{0}^{t} e^{i s(k_2^2+ 2k_1k_2)} dW_{k_2}( s) = \int_{0}^{t} e^{i s k_2^2}(1+\mcO( s k_1k_2)) dW_{k_2}( s).
	\end{equation}
	Then, we would like to solve the integral 
	\begin{equation}\label{eq:stoch_integral}
		\int_{0}^{t} e^{i s k_2^2} dW_{k_2}( s),
	\end{equation} 
	and since \[ \int_{0}^{t} \mcO( s k_1k_2) dW_{k_2}( s) \sim \mcO(t^{3/2} k_1 k_2),\]
	this would be the leading error term for this approximation. However, there is no path-wise solution for \eqref{eq:stoch_integral} and as such we must Taylor expand the entire operator, which gives, 
	\begin{equs}
		\label{discretisation_T_2}
		\begin{aligned}
			\int_{0}^{t} e^{i s(k_2^2+ 2k_1k_2)} dW_{k_2}( s) &=\int_{0}^{t} \left( 1+\mcO( s k_1k_2^2) \right) dW_{k_2}( s)\\
			&=W_{k_2}(t)-W_{k_2}(0 ) + \mcO(t^{3/2}k_1k_2^2).
		\end{aligned}
	\end{equs}

	Plugging this back into \eqref{eq:stoch-taylor_exp} and rewriting into physical space gives us the identity
	\begin{equation}
		\int_{0}^{t}e^{-is\Delta}\left(e^{is\Delta}v\right)\Phi(dW(s))= \Phi(W(t))-\Phi(W(0)) + \mcO(t^{3/2}\grad u\lap\Phi).
	\end{equation}
	\begin{remark}
		We must assume that $ \mathrm{Tr}((\Delta \Phi)^2) < + \infty $ to ensure that we have enough regularity for the two derivatives on the noise coming from the full Taylor expansion.
	\end{remark}
	At the level of decorated trees, one has
	\begin{equs}
		T_2 =  \begin{tikzpicture}[node distance = 14pt, baseline=7]	
			\node[inner] (root) at (0,0) {};	
			\node[inner] (inner4) [above of = root] {};	
			\node[leaf] (leaf1) [above left of = inner4] {$ k_1 $};	
			\node[leaf] (leaf2) [above right of = inner4] {$ k_2 $};	
			\draw[int] (root.north) -- (inner4.south);	
			\draw[Phi] (inner4.north east) -- (leaf2.south west);	
			\draw[positive] (inner4.north west) -- (leaf1.south east);	 		
		\end{tikzpicture}, \quad	(\Pi T_2)(t)= -i  \int_{0}^{t} e^{i s(k_2^2+ 2 k_1k_2)} \Phi_{k_2}dW_{k_2}( s).
	\end{equs}
	With a quick computation, one also realises that
	\begin{equs}
		S(T_2) = 1, \quad \Upsilon(T_2)(v) = {v}_{k_1}.
	\end{equs}
	By setting $ n,r=1 $ and using \eqref{discretisation_T_2} we can obtain
	\begin{equs}
		(\Pi^{n,r} T_2 )(t) =-i\Phi_{k_2} (W_{k_2}(t)-W_{k_2}(0 )).
	\end{equs}
	It remains to treat the double stochastic integral given by:
	\begin{equs}
		T_3 & = \begin{tikzpicture}[node distance = 14pt, baseline=14]
			\node[inner] (root) at (0,0) {};
			\node[inner] (inner1) [above of = root] {};
			\node[inner] (inner2) [above left of = inner1] {};
			\node[inner] (inner3) [above of = inner2] {};
			\node[leaf] (leaf1) [above left of = inner3] {$ k_1 $};
			\node[leaf] (leaf2) [above right of = inner3] {$ k_2 $};
			\node[leaf] (leaf3) [above right of = inner1] {$ k_3 $};
			\draw[int] (root.north) -- (inner1.south);
			\draw[positive] (inner1.north west) -- (inner2.south east);
			\draw[int] (inner2.north) -- (inner3.south);
			\draw[positive] (inner3.north west) -- (leaf1.south east);
			\draw[Phi] (inner3.north east) -- (leaf2.south west);
			\draw[Phi] (inner1.north east) -- (leaf3.south west); 
		\end{tikzpicture}, \quad  	S(T_3) = 1, \quad \Upsilon(T_3)(v) = {v}_{k_1},
		\\
		(\Pi T_3)(t)  & = -i	\int_0^t e^{i  sk^2}\left(e^{-i s k_{12}^2} (\Pi T_2)(s)\right)\Phi_{k_3}dW_{k_3}(s)
	\end{equs}
	where $ k = k_1 + k_2 + k_3 $ and $ k_{12} = k_1 + k_2 $. One first observes, as before, that one has
	\begin{equs}
		(\Pi T_3)(t) = -i	\int_0^t e^{is (k_3^2 + 2k_{12}k_3)} (\Pi T_2)(s)\Phi_{k_3}dW_{k_3}(s).
	\end{equs}
	By performing the Taylor expansion on all terms of the form  $ e^{i s \cdot} $, one gets
	\begin{equs}
		(\Pi^{n,r} T_3 )(t) =  \int_0^t \Phi_{k_2} (W_{k_2}(s)-W_{k_2}(0 )) \Phi_{k_3}dW_{k_3}(s),
	\end{equs}
	with a local error of order $ \mcO(t^{3/2}k_1k_2^2k_3^3) $.
	
	\begin{proposition} \label{expansion_scheme_NLS}
		A low regularity scheme for stochastic NLS with multiplicative noise \eqref{NLS_duhamel} of order $ \mathcal{O}(t^{3/2}) $ is given in Fourier space by:
		\begin{equs}
			U^{n,r}_k(v,t) & = 	e^{-i t k^2}v_k - \sum_{k = -k_1 + k_2 + k_3}	e^{-i t k^2}\frac{e^{2ik_1^2 t}-1}{2k_1^2} \bar{v}_{k_1} v_{k_2} v_{k_3}  \\
			& - \sum_{k= k_1 + k_2} i  e^{-i t  k^2} \Phi_{k_2} (W_{k_2}(t)-W_{k_2}(0 )) v_{k_1} \\
			& - \sum_{k= k_1 + k_2 + k_3}e^{-i t k^2} \int_0^t \Phi_{k_2} (W_{k_2}(s)-W_{k_2}(0 )) \Phi_{k_3}dW_{k_3}(s)  v_{k_1}
		\end{equs}
		where one has to assume $ v $ to be in $ H^1 $ and that $ \mathrm{Tr}\left((\Delta \Phi)^2\right) < + \infty $ .
	\end{proposition} 
	\begin{proof} One can observe that
		\begin{equs}
			\CT^{1,k}_{0}(R) = \lbrace \mathcal{I}_{(\mathfrak{t}_1,0)}(\lambda_k), \, \mathcal{I}_{(\mathfrak{t}_1,0)}(\lambda_k T_1), \, \mathcal{I}_{(\mathfrak{t}_1,0)}(\lambda_k T_2), \, \mathcal{I}_{(\mathfrak{t}_1,0)}(\lambda_k T_3), \,  k_1, k_2, k_3 \in \mathbb{Z}^d \rbrace 
		\end{equs}
		Then, one has the following identity:
		\begin{equs}
			U_{k}^{1}(t, v) &=   \sum_{T \in \CT^{1,k}_{0}(R)} \frac{\Upsilon( T)(v)}{S(T)} (\Pi   T )(t) 
			\\ &= 	e^{-i t k^2}v_k + e^{-itk^2} \sum_{k=-k_1 + k_2 + k_3} \frac{\Upsilon( T_1)(v)}{S(T_1)} (\Pi   T_1 )(t) 
			\\ & + \sum_{k=k_1 + k_2} \frac{\Upsilon( T_2)(v)}{S(T_2)} (\Pi   T_2 )(t) 
			+ \sum_{k=k_1 + k_2 + k_3} \frac{\Upsilon( T_3)(v)}{S(T_3)} (\Pi   T_3 )(t) .
		\end{equs}
		One just has to replace $ (\Pi T_i)(t) $ by its discretisation $ (\Pi^{1,1} T_i)(t) $ in order to conclude.
	\end{proof}

	\subsection{Limitation for higher order}
	Here, an iterated integral that can not be approximated using resonance analysis is exhibited. One observes that our inability to treat integrals such as \eqref{eq:stoch_integral} is due to limitations in calculus of random variables. This is the reason that the resonance approach fails to provide higher order low regularity schemes, which breaks down when we run into integrals of the form:
	\begin{equation}\label{eq:bad_integral}
		\int_{0}^{t} e^{isF(k) }W_k( s)d s = \frac{e^{isF(k) }}{iF(k)}W_k( s)\bigg|_{s=0}^{s=t} - \int_{0}^{t}\frac{e^{isF(k) }}{iF(k)}dW_k( s),
	\end{equation}
	which is the integration by parts formula. In the case of the SNLS this occurs when deriving the order $ \mcO(t^2) $ scheme. If we were to continue to iterate the form \eqref{picard32} by plugging the first stochastic term into one of the initial conditions in the nonlinear term then we would have the integral 
	\begin{align*}
		&\int_{0}^{t}e^{-is\Delta}\left(e^{-is\Delta}\bar{v}\right)\left(e^{is\Delta}v\right)\left( e^{is\Delta}\int_{0}^{s}e^{-is_1\Delta}\left(e^{is_1\Delta}v\right)\Phi dW(s_1)\right)d s\\
		&\quad=\sum_{k=k_1 + k_2 - k_3 + k_4} e^{ikx} v_{k_1}\bar{v}_{k_3}v_{k_4} \Phi_{k_2}\int_{0}^{t} e^{i s P_1}\int_{0}^{s} e^{i s_1 P_2 }dW_{k_2}(s_1)d s,
	\end{align*}
	where $ P_1 = 2 k_3^2-2 k_3 ( k_1 + k_2 + k_4) + 2 k_1 (k_2+ k_4) + 2 k_2 k_4 $ and $ P_2=k_2^2+2k_1k_2 $. We perform the approximation of the stochastic integral by Taylor expansion of the operator, i.e.
	\begin{equation}
		\int_{0}^ s e^{i s_1 P_2 }dW_{k_2}(s_1) = W_{k_2}(s) - W_{k_2}(0)+\mcO( s^{3/2} k_1k_2^2),
	\end{equation}
	which leads to the following,
	\begin{equs}
		\int_{0}^{t} e^{i sP_1(k) }\int_0^s e^{i  s_1 P_2 }dW_{k_2}(s_1)d s=\int_0^t e^{i  s P_1} \left( W_{k_2}(s) - W_{k_2}(0) \right. \\ \left. +\mcO( s^{3/2} k_1k_2^2) \right) d s.
	\end{equs}
	To proceed as in the deterministic setting we would observe that the only part of $ P_1$ corresponding to a second order differential operator is $ 2k_3^2 $ which would lead us to consider the integral
	\begin{equation}
		\int_0^t e^{2i s k_3^2 }\left(W_{k_2}(s) - W_{k_2}(0) \right)d s.
	\end{equation}
	But this has no path-wise solution and we are thus forced to Taylor expand the operator $ e^{i s P_1 } = 1 + \mathcal{O}( s P_1 ) $. This gives us the approximation 
	\begin{align*}
		\int_{0}^{t} e^{i s P_1} \left( W_{k_2}(s) -  W_{k_2}(0)\right)d s &= t  \left( W_{k_2}( s) -  W_{k_2}( 0) \right)\\
		&\quad + \int_0^t s dW_{k_2}( s) + \mcO\rbr{ t^{5/2}P_1},
	\end{align*}
	which implies order $ \mcO(t^{5/2}\lap u) $. Below, we explain how this iterated integral can be encoded via the tree formalism: 
	\begin{equs}
		T_4 & = \begin{tikzpicture}[node distance = 14pt, baseline=7]				
			\node[inner] (root) at (0,0) {};			
			\node[inner] (inner4) [above of = root] {};		
			\node[inner] (inner5) [above of = inner4] {};		
			\node[inner] (inner6) [above of = inner5] {};		
			\node[leaf] (leaf1) [above left of = inner4] {$ k_3 $};		
			\node[leaf] (leaf2) [above right of = inner6] {$ k_2 $};		
			\node[leaf] (leaf4) [above left of = inner6] {$ k_1 $};
			\node[leaf] (leaf3) [above right of = inner4] {$ k_4 $};		
			\draw[int] (root.north) -- (inner4.south);
			\draw[positive] (inner4.north) -- (inner5.south);
			\draw[int] (inner5.north) -- (inner6.south);
			\draw[conj] (inner4.north west) -- (leaf1.south east);
			\draw[positive] (inner4.north east) -- (leaf3.south west);	
			\draw[Phi] (inner6.north east) -- (leaf2.south west);	
			\draw[positive] (inner6.north west) -- (leaf4.south east);			  
		\end{tikzpicture} , \quad S(T_4) = 1, \quad \Upsilon(T_4)(v) = 2 v_{k_1} \bar{v}_{k_3} v_{k_4}
		\\ (\Pi T_4)(t) &  = 	-   \int^{t}_0 e^{i s k^2} e^{i s k_3^2}  e^{-i s k_{12}^2} (\Pi T_2)(s)  e^{-i s k_4^2}d s.
	\end{equs}
	where $ k = k_1 + k_2 -k_3 + k_4 $ and $ k_{12} = k_1 + k_2 $. Using the previous approximation, one can set for $ r \geq 2 $
	\begin{equs} \label{def_Pi_T4}
		\begin{aligned}
			(\Pi^{n,r} T_4)(t) & = -   \sum_{\ell  \leq r-2} \int_{0}^{t} \frac{s^{\ell}}{\ell !} i^{\ell}(P_1 + P_2)^{\ell}  \Phi_{k_2} \left(W_{k_2}(s) -  W_{k_2}(0)\right)d s.
		\end{aligned}
	\end{equs}
	One deduces the following proposition:
	\begin{proposition} \label{counter_example}
		The discretisation of $ (\Pi T_4)(t) $ is given by 
		$ 	(\Pi^{n,r} T_4)(t) $ (defined by \eqref{def_Pi_T4}) for every $ r \geq 2 $. One has
		\begin{equs}
			(\Pi T_4)(t) - (\Pi^{n,r} T_4)(t) = \mathcal{O}(\Phi_{k_2} k_2^{2r-2} k_3^{2r-2} t^{r + 1/2}).
		\end{equs}
	\end{proposition}
	As a corollary, we obtain the local error of Theorem \ref{main_theorem_general} when $ r \geq 2 $.
	Hence we see that for schemes of order $ \mcO(t^{5/2}) $ and higher we will no longer be able to obtain regularity lower than in traditional schemes since the stochastic integral will prevent us from performing the resonance analysis.  For higher orders schemes, we will also have integrals of type $ \int_0^t s^n \Phi dW(s) $ which have no path-wise solution and would need to be approximated. This can be done with schemes such as those described in \cite{K.P1992,R2010}. 
	
	\section{Error analysis}
	\label{Sec::4}
	In this section the definitions of convergence, error, and stability are introduced. We then proceed to state the main results of the paper, a summary of which is as follows:
	\begin{itemize}
		\item Proposition \ref{prop:schrodinger_O1/2_local} and Theorem \ref{scheme_order_2} which are local error bounds for the schemes \ref{scheme:schrodinger_O1/2} and \ref{scheme:schrodinger_O1} which are of order $ \mcO(t) $ and $ \mcO(t^{3/2}) $ respectively. These schemes can also be applied to Schr\"odinger's equation with additive noise as a special case.
		\item Proposition \ref{Manakov_local_error} which is a local error bound for the low regularity scheme \ref{scheme:manakov_O1/2} of order $ \mcO(t) $ for the Manakov equations. 
		\item Propostion \ref{lem:stability}, which is a statement about the stability of the schemes.
		\item Theorem \ref{prop:global_error}, which states the rates of convergence of all the schemes presented above.  
	\end{itemize}
	The statement of each theorem is followed by a proof of the local error in each case. We then present a proof of the stability and global error, which works for all four schemes. In each case we obtain results for both strong and pathwise convergence. 

	\subsection{Definitions of convergence and error}\label{ss::notation}
	We measure error in terms of the Sobolev norm
	\begin{equation}
		\norm{f}_{k,p}^p = \sum_{\ell=0}^k\norm{D^{(\ell)}f}_p^p,
	\end{equation}
	where $ \norm{\cdot}_p $ is the $ L^p $ norm and $ D^{\ell} $ represents the $ \ell^{th} $-order weak differential operator. We will use the following notation 
	\begin{itemize}
		\item We recall the notation $ u(0) = v $ for the initial data. If we are considering two initial datum then we will use $ u_1(0) = v_1 $ and $ u_2(0) = v_2 $.
		\item Let $ \pi(t_0,T) $ be a uniform discretisation of the interval $[t_0,t_N] $ such that $ t_N = T $, $ t_0 < t_1<...<t_N $,  and $ t_{\ell+1} - t_\ell = t \ \ \forall \ell\in\set{0,\dots,N} $.
		\item $ v(t_\ell) = E_t(v(t_{\ell-1})) $ is the exact value of the solution at time $ t_\ell $ given by Duhamel's formulation.
		\item $v_N =  S_t(v_{N-1})=S_t^N(v)$  is the approximation given by the $ N$th  iteration of the scheme. This can also be written as \[\prod_{i=0}^{N-1}S_t(u_i) = S_t(v_{N-1})\circ S_t(v_{N-2})\circ \cdots \circ S_t(v),\]and $ \set{v_n}_{n=0}^N $ is a sequence of approximations beginning with $ v $.
	\end{itemize}
	
	\begin{definition}[Local error]\label{defn:local-error}
		A scheme is said to have local error of order $ \gamma $ if
		\begin{equation}\label{eq:pathwiselocal}
			\norm{E(u(t_\ell), t) - S(u(t_\ell), t)}^p_{k,p} = \norm{u(t_{\ell+1}) - \tilde{u}(t_{\ell+1})}^p_{k,p} \leq Ct^{\gamma p}.
		\end{equation}	
	\end{definition}
	\begin{remark}
		Note that $ k $ refers to the regularity required by the initial data $ v $. This regularity is necessary to accommodate the Taylor expansion of the integrands in the iteration of Duhamel's formula. Hence it will depend on the order of the scheme, higher orders corresponding to higher regularity. The impact of Taylor expansion is reduced through resonance analysis and we obtain higher order schemes than in previous works but with lower regularity assumptions. 
	\end{remark}
	\begin{definition}[Global error]\label{defn:global-error}
		A scheme is said to have \textit{strong} global error of order $ \alpha $ on the interval $ [0,T] $ if 
		\begin{equation}
			\mbbE\sqbr{\sup_{t_i\in\pi(t_0,T)}\norm{u(t_i) - S_t(u_{i-1})}^p_{k,p}} \leq Ct^{\alpha p} 
		\end{equation}
		and \textit{path-wise} global error of order $ \alpha $ if 
		\begin{equation}
			\sup_{t_i\in\pi(t_0,T)}\norm{u(t_i) - S_t(u_{i-1})}^p_{k,p} \leq Ct^{\alpha p}
		\end{equation}
	\end{definition}
	\begin{remark}
		In the literature the most commonly seen variant is strong convergence which is a bound on the moments of the error. In practice, simulations are performed path-wise and this kind of convergence is the strongest. Weak convergence studies the convergence in moments, this is useful in the study of the ergodicity and stationary distributions of stochastic systems. It is proved in \cite{K.N2007} that strong convergence of order $ \mathcal{O}(t^\alpha) $ implies path-wise convergence or order $ \mathcal{O}(t^{\alpha-\eps}) $ for some arbitrary $ \eps $. 
	\end{remark}
	\begin{definition}[Stability and consistency]\label{def:stabitlity_and_consistancy}
		A scheme is said to be consistent if the order of the local error is strictly greater than $ 0 $, see definition 4.1 in \cite{T2012}. 
		It is said to be stochastically numerically stable, see \cite{K.P1992}, if for any finite interval $ [t_0,T] $ there exists a constant $ C>0 $, such that for each $ \eps>0 $ and $ t\in(0,C) $ it holds that
		\begin{equation*}
			\lim_{\norm{v_1 - v_2}\rightarrow 0}\sup_{t\in(t_0,T)}\mbbP\rbr{\norm{S_t^N(v_1) - S^N_t(v_2)}^p_{k,p}\geq \eps} =0.
		\end{equation*}
	\end{definition}
	\begin{remark}
		The definition of stochastic stability differs slightly from the deterministic setting wherein a scheme is stable if and only if there exists a constant $ C_t >0$ such that for all pairs of initial datum $ (v_1,v_2) $ the following holds
		\begin{equation}
			\norm{S^N_t(v_1)-S^N_t(v_2)}\leq C_t\norm{v_1-v_2}.
		\end{equation}
		Stochastic numerical stability follows immediately by application of a Martingale inequality followed by Markov's inequality. We refer to \cite{T2012} for definitions concerning numerical schemes for deterministic PDE.
	\end{remark}
	Gradients appearing on the iterations of a scheme itself can cause instability. To take care of this we could use a numerical differentiation strategy, but this normally breaks the martingale structure of the noise discretisation and introduces complexity to the analysis. Filter functions, which are described in detail in \cite{H.W.L2002} allow us to ensure stability by considering a function that preserves the local error while bounding the derivative. 
	\begin{definition}[Filter functions]\label{def:filter_function}
		A filter function is an operator of the form 
		\begin{equation}
			\Psi = \Psi(it\mcL),\quad\Psi(0) = 1,\quad \norm{\Psi(it \mcL) u}\leq 1.
		\end{equation}	
	\end{definition}
	\begin{remark} 
		A filter as defined above will preserve the local error without introducing regularity on the initial conditions as long as the additional condition is met that 
		\begin{equ}
			\Psi(it\mcL)u =\mcL u + \mcO(t\mcL u) 
		\end{equ}
		and $ \mcL $ does not require more regularity than the scheme in question.
	\end{remark}
	\subsection{Schr\"odinger's equation}

	We have described the derivation of this scheme in the previous section. 
	\begin{scheme}\label{scheme:schrodinger_O1/2}
		Define $f_{1}(\sigma)=\frac{e^{\sigma }-1}{\sigma}$, 
		\begin{align*}
			u_{\ell+1}&=e^{it\Delta}u_{\ell}-it e^{it\Delta}\left(u_{\ell}^{2}f_{1}\left(2i\Delta\right)\bar{u}_{\ell}\right)-i\sqrt{t} e^{it\Delta}u_{\ell} \Phi\chi_\ell \\
			&\quad - \frac{t}{2}e^{it\Delta}u_\ell\rbr{\Phi\chi^2_\ell - \mathrm{Tr}(\Phi^2)}.
		\end{align*} 
	\end{scheme}
	\begin{proposition}\label{prop:schrodinger_O1/2_local}
		For initial condition $ v\in H^1 $ and $ \mathrm{Tr}(
		(\lap\Phi)^2)< + \infty $, the scheme \ref{scheme:schrodinger_O1/2} approximates \eqref{schrodinger_equation} with pathwise local error of order $\mathcal{O}(t^{3/2})$.
	\end{proposition} 
	\begin{proof}
		The theorem is just a rewriting of the scheme given in Proposition~\ref{expansion_scheme_NLS} in physical space. We use \eqref{eq:discrete_W} and \eqref{eq:discrete_W^2} for the rewriting of the stochastic iterated integrals.
	\end{proof}
	
	\begin{remark} In the case of additive noise, one has some simplification in the scheme due to the fact that one cannot get the double stochastic iterated integral. The scheme has the following form:
		\begin{align*}
			u_{\ell+1}&=e^{it\Delta}u_{\ell}-it e^{it\Delta}\left(u_{\ell}^{2}f_{1}\left(2i\Delta\right)\bar{u}_{\ell}\right)-i\sqrt{t} e^{it\Delta} \Phi\chi_\ell.
		\end{align*}
		The scheme estimates the solution with local error of order $\mathcal{O}(\nabla u,\Delta\Phi,t^{3/2})$.
	\end{remark}
	
	If we want to go up to order $ \mcO\rbr{t^2} $ in the discretisation, we need to collect more iterated integrals represented by new decorated trees. 
	One has for $ r= \frac{3}{2} $
	\begin{equs}
		\CT^{\frac{3}{2},k}_{0}(R) =  \left \lbrace   \mathcal{I}(\lambda_k), \,  \mathcal{I}(\lambda_k T_i), \, k_j \in \mathbb{Z}^d \, i \in \lbrace 1,...,7 \rbrace, \, j \in \lbrace 1,2,3,4 \rbrace \right \rbrace
	\end{equs}
	where 
	\begin{equs}
		T_5  & = \begin{tikzpicture}[node distance = 14pt, baseline=7]				
			\node[inner] (root) at (0,0) {};			
			\node[inner] (inner4) [above of = root] {};		
			\node[inner] (inner5) [above of = inner4] {};		
			\node[inner] (inner6) [above of = inner5] {};		
			\node[leaf] (leaf1) [above left of = inner4] {$ k_3 $};		
			\node[leaf] (leaf2) [above right of = inner6] {$ k_2 $};		
			\node[leaf] (leaf4) [above left of = inner6] {$ k_1 $};
			\node[leaf] (leaf3) [above right of = inner4] {$ k_4 $};		
			\draw[int] (root.north) -- (inner4.south);
			\draw[conj] (inner4.north) -- (inner5.south);
			\draw[conj-int] (inner5.north) -- (inner6.south);
			\draw[positive] (inner4.north west) -- (leaf1.south east);
			\draw[positive] (inner4.north east) -- (leaf3.south west);	
			\draw[Phi-conj] (inner6.north east) -- (leaf2.south west);	
			\draw[conj] (inner6.north west) -- (leaf4.south east);			  
		\end{tikzpicture} , \quad S(T_5) = 2, \quad \Upsilon(T_5)(v) = 2 \bar{v}_{k_1} v_{k_3} v_{k_4}
		\\ (\Pi T_5)(t) &  = 	- i  \int^{t}_0 e^{i s k^2} e^{-i s k_3^2}  e^{i s k_{12}^2} \overline{(\Pi T_2)(s)}  e^{-i s k_4^2}d s,
	\end{equs}
	with $ k = -k_1 -k_2 + k_3 + k_4 $. The decorated tree $ T_5 $ is very similar to $ T_4 $ and it is of order $ \mathcal{O}(t^{3/2}) $.
	\begin{equs}
		T_6 & = \begin{tikzpicture}[node distance = 14pt, baseline=14]
			\node[inner] (root) at (0,0) {};
			\node[inner] (inner1) [above of = root] {};
			\node[inner] (inner2) [above left of = inner1] {};
			\node[inner] (inner3) [above of = inner2] {};
			\node[leaf] (leaf1) [above left of = inner3] {$ k_1 $};
			\node[leaf] (leaf2) [above right of = inner3] {$ k_3 $};
			\node[leaf] (leaf4) [above of = inner3] {$ k_2 $};
			\node[leaf] (leaf3) [above right of = inner1] {$ k_4 $};
			\draw[int] (root.north) -- (inner1.south);
			\draw[positive] (inner1.north west) -- (inner2.south east);
			\draw[int] (inner2.north) -- (inner3.south);
			\draw[conj] (inner3.north west) -- (leaf1.south east);
			\draw[positive] (inner3.north east) -- (leaf2.south west);
			\draw[positive] (inner3.north) -- (leaf4.south);
			\draw[Phi] (inner1.north east) -- (leaf3.south west); 
		\end{tikzpicture}, \quad  	S(T_6) = 2, \quad \Upsilon(T_6)(v) = 2 \bar{v}_{k_1} v_{k_2} v_{k_3},
		\\
		(\Pi T_6)(t)  & = -i	\int_0^t e^{i  sk^2}\left(e^{-i s k_{123}^2} (\Pi T_1)(s)\right)\Phi_{k_4}dW_{k_4}(s)
	\end{equs}
	where $ k = -k_1 + k_2 + k_3 + k_4 $ and $ k_{123} = k_1 + k_2 + k_3 $. 
	The last decorated tree of order $ \mathcal{O}(t^{3/2}) $ is given by
	\begin{equs}
		T_7 & = \begin{tikzpicture}[node distance = 14pt, baseline=14]
			\node[inner] (root) at (0,0) {};
			\node[inner] (inner1) [above of = root] {};
			\node[inner] (inner2) [above left of = inner1] {};
			\node[inner] (inner3) [above of = inner2] {};
			\node[inner] (inner11) [above left of = inner3] {};
			\node[inner] (inner111) [above  of = inner11] {};
			\node[leaf] (leaf1) [above left of = inner111] {$ k_1 $};
			\node[leaf] (leaf11) [above right of = inner111] {$ k_2 $};
			\node[leaf] (leaf2) [above right of = inner3] {$ k_3 $};
			\node[leaf] (leaf3) [above right of = inner1] {$ k_4 $};
			\draw[int] (root.north) -- (inner1.south);
			\draw[positive] (inner1.north west) -- (inner2.south east);
			\draw[int] (inner2.north) -- (inner3.south);
			\draw[int] (inner11.north) -- (inner111.south);
			\draw[positive] (inner111.north west) -- (leaf1.south east);
			\draw[Phi] (inner111.north east) -- (leaf11.south west);
			\draw[positive] (inner3.north west) -- (inner11.south east);
			\draw[Phi] (inner1.north east) -- (leaf3.south west); 
			\draw[Phi] (inner3.north east) -- (leaf2.south west); 
		\end{tikzpicture}, \quad  	S(T_7) = 1, \quad \Upsilon(T_7)(v) = {v}_{k_1},
		\\
		(\Pi T_7)(t)  & = -i	\int_0^t e^{i  sk^2}\left(e^{-i s k_{123}^2} (\Pi T_3)(s)\right)\Phi_{k_4}dW_{k_4}(s)
	\end{equs}
	where $ k = k_1 + k_2 + k_3 + k_4 $.
	We first observe that
	\begin{equs}
		(\Pi^{n,1} T_1 )(t) = 	(\Pi^{n,\frac{3}{2}} T_1 )(t). 
	\end{equs}
	Then for the tree $ T_2 $, one has to push the expansion a bit further. Indeed, 
	\begin{equs}
		(\Pi T_2)(t) & =  -i  \int_{0}^{t} e^{i s(k_2^2+ 2 k_1k_2)} \Phi_{k_2}dW_{k_2}( s) \\
		& = -i  \int_{0}^{t} 1 +  i s(k_2^2+ 2 k_1k_2) + \mathcal{O}(s^2(k_2^2+ 2 k_1k_2)^2) \Phi_{k_2}dW_{k_2}(s).
	\end{equs}
	One has to set
	\begin{equs}
		(\Pi^{n,\frac{3}{2}} T_2)(t) =  -i   \Phi_{k_2}(W_{k_2}(t) -  W_{k_2}(0)) - i \int_{0}^{t}   i s(k_2^2+ 2 k_1k_2)  \Phi_{k_2}dW_{k_2}( s)
	\end{equs}
	with an error of order $ \mathcal{O}(s^{5/2} (k_2^2 + 2 k_1 k_2)^2 \Phi_{k_2}) $. Then, the term $ \frac{\Upsilon( T_2)(v)}{S(T_2)} (\Pi^{n,\frac{3}{2}}   T_2 )(t) $ back in physical space is given by
	\begin{equs}
		e^{it\lap}\rbr{-i\Phi W(t)v +  v \int_0^t s \lap\Phi dW(s) + 2 \grad v \int_0^t s \grad \Phi dW(s) }.
	\end{equs}
	\begin{remark}
		The integral $ \int_0^t s \lap\Phi dW(x,s) $ cannot be solved analytically but integrals like this can be approximated by other schemes and it will not affect the regularity of the initial data. Hence, here we assume that this integral can be approximated up to the order of the remainder term in \eqref{picard32}.
	\end{remark}
	For the iterated integral corresponding to $ T_3 $, one has
	\begin{equs}
		(\Pi T_3)(t)  & = -i	\int_0^t e^{i  sk^2}\left(e^{-i s k_{12}^2} (\Pi T_2)(s)\right)\Phi_{k_3}dW_{k_3}(s).
	\end{equs}
	As we have two stochastic iterated integrals, one can replace $  (\Pi T_2)(s) $ by its discretisation up to order $ \mathcal{O}(s^{3/2}) $. Then
	\begin{equs}
		(\Pi T_3)(t)  & = -i	\int_0^t e^{i  sk^2}\left(e^{-i s k_{12}^2} (\Pi^{n,1} T_2)(s)\right)\Phi_{k_3}dW_{k_3}(s) 
		\\ & - i \int_0^t e^{i  sk^2}\left(e^{-i s k_{12}^2} \mathcal{O}(s^{3/2}k_1k_2^2 ) \right)\Phi_{k_3}dW_{k_3}(s)
		\\ & = -	\int_0^t e^{i  s (k_3^2 + 2 k_3 k_{12})}\left( \Phi_{k_2} (W_{k_2}(t)-W_{k_2}(0 ))\right)\Phi_{k_3}dW_{k_3}(s) 
		+ \mathcal{O}(t^{2}k_1k_2^2 ).
	\end{equs}
	Then, one proceeds with a full Taylor expansion of the term $ e^{is ...} $ in the first integral to get:
	\begin{equs}
		(\Pi T_3)(t)  & = -	\int_0^t \left( \Phi_{k_2} (W_{k_2}(t)-W_{k_2}(0 ))\right)\Phi_{k_3}dW_{k_3}(s) 
		+ \mathcal{O}(t^{\frac{5}{2}}k_3^2k_{12}^2 )  +\mathcal{O}(t^{2}k_1k_2^2 ).
	\end{equs}
	In the end, one can fix
	\begin{equs}
		(\Pi^{n,\frac{3}{2}} T_3)(t) =  -	\int_0^t \left( \Phi_{k_2} (W_{k_2}(t)-W_{k_2}(0 ))\right)\Phi_{k_3}dW_{k_3}(s) 
	\end{equs}
	and it is an approximation of $ (\Pi T_3)(t) $ with an error $ \mathcal{O}(t^{2}k_1k_2^2 ) $ asking one derivative on the initial data and two derivatives on $ \Phi $.
	For the tree $ T_4 $, one has from \eqref{def_Pi_T4}
	\begin{equs}
		(\Pi^{n,\frac{3}{2}} T_4)(t) & = -  i   \int_{0}^{t} \Phi_{k_2} \left(W_{k_2}(s) -  W_{k_2}(0)\right)d s
	\end{equs}
	with an error of order $ \mathcal{O}(t^{5/2} k_2^2 k_3^2) $.
	A similar computation for $ (\Pi T_5)(t) $ gives:
	\begin{equs}
		(\Pi^{n,\frac{3}{2}} T_5)(t) & =   i   \int_{0}^{t} \Phi_{k_2} \left(\bar{W}_{k_2}(s) -  \bar{W}_{k_2}(0)\right)d s.
	\end{equs}

	The computation of the discretisation of $ (\Pi T_6)(t) $ and $ (\Pi T_7)(t) $ works as the same as for $  (\Pi T_3)(t) $, one uses approximations of $ (\Pi T_1)(t) $ and $ (\Pi T_3)(t) $ from previous order. The main difference is that now we suppose $ n=2 $ which gives different approximations. Indeed, one has
	\begin{equs}
		(\Pi T_1)(t)  = - i t  + \mathcal{O}(t^2 k_1^2), \quad (\Pi^{n,1} T_1)(t)  = - i t.
	\end{equs}
	One does not need anymore to peform a low regularity expansion as we have assumed that the initial data is in $ H^2 $. Therefore, a full Taylor expansion is enough for approximating $ (\Pi T_1)(t) $.
	A short computation gives for  $ (\Pi T_6)(t) $
	\begin{equs}
		(\Pi T_6)(t)  & = -i	\int_0^t e^{i  sk^2}\left(e^{-i s k_{123}^2} (\Pi^{n,1} T_1)(s)\right)\Phi_{k_4}dW_{k_4}(s) + \mathcal{O}(  t^{\frac{5}{2}} k_1 k_2 k_3)
		\\ &  = -	\int_0^t s \Phi_{k_4}dW_{k_4}(s) + \mathcal{O}(  t^{\frac{5}{2}} k_4^2 k_1) + \mathcal{O}(  t^{\frac{5}{2}} k_1^2)
	\end{equs}
	We have
	\begin{equs}
		(\Pi^{n,\frac{3}{2}} T_6)(t) = 
		-	\int_0^t s \Phi_{k_4}dW_{k_4}(s).
	\end{equs}
	It remains to perform the computation for $ (\Pi T_7)(t) $
	\begin{equs}
		(\Pi T_7)(t)  & = -i	\int_0^t e^{i  sk^2}\left(e^{-i s k_{123}^2} (\Pi T_3)(s)\right)\Phi_{k_4}dW_{k_4}(s)
		\\ & = -i	\int_0^t e^{i  sk^2}\left(e^{-i s k_{123}^2} (\Pi^{n,1} T_3)(s)\right)\Phi_{k_4}dW_{k_4}(s) + \mathcal{O}(t^{\frac{5}{2}}  k_2^2 k_3^2 )
		\\ & = -i	\int_0^t (\Pi^{n,1} T_3)(s) \Phi_{k_4}dW_{k_4}(s) + \mathcal{O}(t^{\frac{5}{2}}  k_4^2 )  + \mathcal{O}(t^{\frac{5}{2}}  k_2^2 k_3^2 ).
	\end{equs}
	Therefore
	\begin{equs}
		(\Pi^{n,\frac{3}{2}} T_7)(t) = -i	\int_0^t (\Pi^{n,1} T_3)(s) \Phi_{k_4}dW_{k_4}(s).
	\end{equs}
	\begin{scheme}\label{scheme:schrodinger_O1}
		This derivation leads us to the following discretisation of Schr\"odinger's equation \ref{schrodinger_equation}.
		\begin{equs}
			S_t(u_\ell) = u_{\ell+1}&=e^{it\Delta}u_{\ell}-it e^{it\Delta} u_{\ell}^{2}\bar{u}_{\ell}
			-i\sqrt{t} e^{it\Delta} u_{\ell} \Phi\chi_{\ell} \\ &\quad   + e^{i t \Delta} \left(  u_{\ell} \int_0^t s \lap\Phi dW(s) + 2 \Psi_1(i t \grad) u_{\ell} \int_0^t s \grad \Phi dW(s) \right)  \\
			& \quad - \frac{t}{2}e^{it\Delta}u_\ell\rbr{\Phi\chi^2_{\ell} - 1} 
			- (1+2 i)t^{3/2}  e^{it\Delta} u_{\ell}^2 \bar{u}_{\ell}\Phi\chi_{\ell} 
			\\ &\quad  +  i e^{it \Delta} u_{\ell} \bar{u}_{\ell}^2  \int_{0}^{t} \Phi \left(\bar{W}(s) -  \bar{W}(0)\right)d s 
			\\ &\quad + i t^{3/2}e^{i t \Delta} u_{\ell} \rbr{\frac16\Phi\chi_{\ell}^3 -  \frac12\text{Tr}(\Phi D\chi_{\ell}\Phi^*)}.
		\end{equs}
		Here $ \Psi_1(it\grad) $ is a filter function defined in \ref{def:filter_function}.
	\end{scheme}
	\begin{theorem} \label{scheme_order_2}
		For $ v \in H^2 $ and $ \Phi$ such that $ \mathrm{Tr}((\Delta^2 \Phi)^2) < + \infty $, the scheme \ref{scheme:schrodinger_O1} approximates \eqref{schrodinger_equation} with $ \sigma=u $ to order $ \mcO(t^2) $ locally.
	\end{theorem}
	\begin{proof}
		From the previous computation, the scheme comes from the following expansion:
		\begin{equs}
			&	U_{k}^{2,\frac{3}{2}}(t, v) \\ & = e^{-itk^2} v_k +   \sum_{ i =1}^7 e^{-itk^2} \frac{\Upsilon( T_i)(v)}{S(T_i)} (\Pi^{2,\frac{3}{2}}   T_i )(t)
			\\ & = e^{-itk^2} v_k  - \sum_{k = -k_1 + k_2 + k_3}	i t e^{-i t k^2} \bar{v}_{k_1} v_{k_2} v_{k_3}  \\
			& - \sum_{k= k_1 + k_2} i  e^{-i t  k^2} ( \Phi_{k_2} (W_{k_2}(t)-W_{k_2}(0 )) v_{k_1} + \int_{0}^{t}  s(k_2^2+ 2 k_1k_2)  \Phi_{k_2}dW_{k_2}( s)) \\
			& + \sum_{k= k_1 + k_2 + k_3}e^{-i t k^2} \int_0^t \Phi_{k_2} (W_{k_2}(s)-W_{k_2}(0 )) \Phi_{k_3}dW_{k_3}(s)  v_{k_1}
			\\ &  - 2 \sum_{k= k_1 + k_2 - k_3 + k_4} i e^{-itk^2}   \int_{0}^{t} \Phi_{k_2} \left(W_{k_2}(s) -  W_{k_2}(0)\right)d s v_{k_1} \bar{v}_{k_3} v_{k_4}
			\\ &  + \sum_{k= -k_1 - k_2 + k_3 - k_4} i e^{-itk^2}   \int_{0}^{t} \Phi_{k_2} \left(\bar{W}_{k_2}(s) -  \bar{W}_{k_2}(0)\right)d s \bar{v}_{k_1} v_{k_3} \bar{v}_{k_4}
			\\ 	& - \sum_{k=-k_1 + k_2 + k_3 + k_4} e^{-it k^2}	\int_0^t s \Phi_{k_4}dW_{k_4}(s) \bar{v}_{k_1} v_{k_3} v_{k_4}
			\\ & + \sum_{k=k_1 + k_2 + k_3 + k_4}i e^{- i t k^2} \int_0^t	\int_0^s \int_0^{s_1}  \Phi_{k_2} dW_{k_2}(s_2)\Phi_{k_3}dW_{k_3}(s_1) \Phi_{k_4}dW_{k_4}(s) v_{k_1}. 
		\end{equs}
		where all the terms have to be mapped to physical space.
		The scheme is an approximation of the following truncated expansion: 
		\begin{equs}
			U_{k}^{\frac{3}{2}}(t, v) = e^{-itk^2} v_k +   \sum_{ i =1}^7 e^{-itk^2} \frac{\Upsilon( T_i)(v)}{S(T_i)} (\Pi   T_i )(t)
		\end{equs}
		The local error is given by the maximum over the following errors:
		\begin{equs}
			(\Pi   T_i )(t) - (\Pi^{2,\frac{3}{2}}   T_i )(t)
		\end{equs}
		which is of order $ \mathcal{O}(t^{2}\Delta^2 \Phi\Delta v)  $ in physical space.
	\end{proof}
	
	\begin{remark}
		To estimate the gradient operator $ \grad $ we can use the filter function 
		\begin{equs}
			\Psi_1(i \grad t) = \frac{e^{it\grad}-1}{it} = \grad + \mcO(t^2\lap).
		\end{equs}
		On the other hand, for the Laplacian, we can use \[ \Psi_2(i \grad t) =  \rbr{\frac{e^{it\grad}-1}{it}}^2 = \lap + \mcO(t\lap). \] In general, as presented in \cite[Sec. 5.2]{B.B.S2022}, we can use the optimal low regularity filter, \[\Psi_p(i \grad t) = \rbr{\frac{e^{it\grad}-1}{it}}^p, \] to approximate $ \grad^p $.
	\end{remark}
	\subsection{Manakov system}
	Previous schemes by Berg, Cohen and Dujardin \cite{B.C.D2020,B.C.D2021} and Gazzeau \cite{G2014} for the Manakov equation require initial conditions such that $ v\in H^5 $  for schemes of order $ \mcO(t^{1/2})$. Here we derive a scheme by the process described in section \ref{Sec::3} and show that we can obtain a globally order $ \mcO(t^{1/2}) $ scheme with $ u\in H^3 $. Our scheme uses Duhamel iterations to obtain terms up to order $ \mcO(t^{3/2}) $ and then stabilisation is achieved via filter functions which replace gradient operators acting upon the numerical solution. Previous schemes work directly with the mild form and use only Taylor expansion combined with an implicit discretisation of the noise, making the analysis and computation significantly harder. 
	
	The scheme for the Manakov equation is given below by
	
	\begin{scheme}[Manakov System]\label{scheme:manakov_O1/2}
		\begin{align*}
			S_t(u_\ell)&=e^{it\Delta}u_{\ell}-it e^{it\Delta} u_{\ell}^{2} \bar{u}_{\ell} -i\sqrt{t}C_\gamma\sum_{n=1}^3\sigma_n  e^{it\lap}\rbr{\Psi_1\rbr{it\grad} u_\ell }\chi_{n,\ell} \\
			&\quad + \frac{t}{2} C^2_\gamma\hspace{-.3cm}\sum_{\substack{n<m\\ n,m\in\set{1,2,3}}}\hspace{-.3cm}\sigma_m\sigma_n  e^{it\lap} \Psi_2(it\grad) u_\ell \sqbr{\frac12\rbr{\chi^2_{n,\ell} - 1} + \chi_{n,\ell}\chi_{m,\ell}}
		\end{align*}
		where the $ \chi_{n,\ell} $ are independent Gaussian random variables,  $ \Psi_1 $, $\Psi_2$ are filter functions as defined in \ref{def:filter_function}.
	\end{scheme} 
	The derivation of this scheme is similar to the scheme for multiplicative NLS as the nonlinearity is the same. One can use the same decorated formalism by setting the Fourier mode equal to zero on the leaves:
	\begin{equs}
		T_1 = 
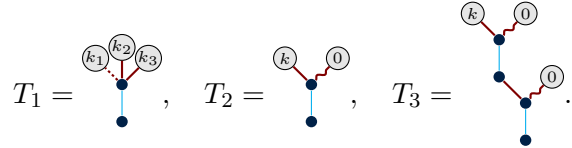
\begin{tikzpicture}[node distance = 14pt, baseline=7]				
			\node[inner] (root) at (0,0) {};			
			\node[inner] (inner4) [above of = root] {};			
			\node[leaf] (leaf1) [above left of = inner4] {$ k_1 $};		
			\node[leaf] (leaf2) [above of = inner4] {$ k_2 $};		
			\node[leaf] (leaf3) [above right of = inner4] {$ k_3 $};		
			\draw[int] (root.north) -- (inner4.south);
			\draw[positive] (inner4.north) -- (leaf2.south);
			\draw[conj] (inner4.north west) -- (leaf1.south east);
			\draw[positive] (inner4.north east) -- (leaf3.south west);		  
		\end{tikzpicture}, \quad	T_2 = \begin{tikzpicture}[node distance = 14pt, baseline=7]	
			\node[inner] (root) at (0,0) {};	
			\node[inner] (inner4) [above of = root] {};	
			\node[leaf] (leaf1) [above left of = inner4] {$ k $};	
			\node[leaf] (leaf2) [above right of = inner4] {$ 0 $};	
			\draw[int] (root.north) -- (inner4.south);	
			\draw[Phi] (inner4.north east) -- (leaf2.south west);	
			\draw[positive] (inner4.north west) -- (leaf1.south east);	 		
		\end{tikzpicture}, \quad  T_3 = \begin{tikzpicture}[node distance = 14pt, baseline=14]
			\node[inner] (root) at (0,0) {};
			\node[inner] (inner1) [above of = root] {};
			\node[inner] (inner2) [above left of = inner1] {};
			\node[inner] (inner3) [above of = inner2] {};
			\node[leaf] (leaf1) [above left of = inner3] {$ k $};
			\node[leaf] (leaf2) [above right of = inner3] {$ 0 $};
			\node[leaf] (leaf3) [above right of = inner1] {$ 0 $};
			\draw[int] (root.north) -- (inner1.south);
			\draw[positive] (inner1.north west) -- (inner2.south east);
			\draw[int] (inner2.north) -- (inner3.south);
			\draw[positive] (inner3.north west) -- (leaf1.south east);
			\draw[Phi] (inner3.north east) -- (leaf2.south west);
			\draw[Phi] (inner1.north east) -- (leaf3.south west); 
		\end{tikzpicture}.
	\end{equs}
	Then, by modifying the definition of $ \Pi $ to sum implicitly over $ m $ to account for the noises we have:
	\begin{equs}
		(\Pi T_1)(t) & = 	- i \int^{t}_0 e^{i s k^2} e^{i s k_1^2}  e^{-i s k_2^2}  e^{-i s k_3^2}d s
		\\
		(\Pi T_2)(t) & =  k \sum_{m=1}^3  \int_{0}^{t} dW_{m}( s)
		\\
		(\Pi T_3)(t) & =  \sum_{n,m\in\set{1,2,3}}
		k \int_0^t \int_0^s d W_m(s_1)  dW_n(s)
	\end{equs}
	where $ k = -k_1 + k_2 + k_3 $ in the first integral.
	The stochastic integrals do not need any discretisation and one can set:
	\begin{equs}
		(\Pi^{n,1} T_2)(t) = 	(\Pi T_2)(t), \quad 	(\Pi^{n,1} T_3)(t) = 	(\Pi T_3)(t).
	\end{equs}
	For the simulation, one can rewrite the first stochastic iterated integral as the same as for NLS using independent Gaussian random variables $ \chi_m $. The second stochastic iterated integral needs a more careful rewriting.
	When $ n=m $, one has
	\begin{equation*}
		\int_0^t \int_0^s   dW_n(s_1) dW_n(s).
	\end{equation*}
	This is a standard iterated Itô integral and can be written as $ \frac{1}{2}(W^2_n(t) - t) $. Next we have the cross terms $ n\neq m $ which have no closed form solution. However, by integration by parts we have the following identity for two independent Wiener processes $ W_n,W_m $,
	\begin{equation*}
		\int_0^t W_m(s) dW_n(s) + \int_0^t W_n(s) dW_m(s) = W_m(t)W_n(t).
	\end{equation*} 
	This means that for the sum over all $ n $ and $ m $ we  have 
	\begin{equ}
		\sum_{n,m\in\set{1,2,3}} \int_0^t \int_0^s   dW_n(s_1) dW_n(s) =\hspace{-0.25cm}\sum_{\substack{n<m\\ n,m\in\set{1,2,3}}}\hspace{-0.25cm}\sqbr{\frac12\rbr{W_n^2(t) - t} + W_n(t)W_m(t)},
	\end{equ}
	If $ n=3 $, one does not need a low regularity approximation for $ (\Pi T_1)(t) $ and can proceed with a full Taylor expansion:
	\begin{equs}
		(\Pi T_1)(t) = - i \int^{t}_0 e^{i s k^2} e^{i s k_1^2}  e^{-i s k_2^2}  e^{-i s k_3^2}d s = - i t + \mathcal{O}(t^2 (k^2 + k_1^2 -k_2^2 -k_3^2))
	\end{equs}
	and therefore, one has
	\begin{equs}
		(\Pi^{n,3} T_1)(t) =  - i t.
	\end{equs}

	\begin{proposition} \label{Manakov_local_error}
		For initial conditions $ v\in H^3 $ the scheme \ref{scheme:manakov_O1/2} has local error of order $ \mcO(t^{3/2}) $.
	\end{proposition}
	\begin{proof}
		We first start with a truncated tree series:
		\begin{equs}
			U_{k}^{1}(t, v) = e^{-itk^2} v_k +   \sum_{ i =1}^3 e^{-itk^2} \frac{\Upsilon( T_i)(v)}{S(T_i)} (\Pi   T_i )(t)
		\end{equs}
		which is a local approximation of the solution up to order $\mathcal{O}(t^{3/2})$ and one needs already an initial data $ v $ in $ H^3$. This is due to the fact that from the third iteration of Duhamel's formulation in the stochastic integral one has three derivatives. These terms are neglected as they  are of order $\mathcal{O}(t^{3/2})$.
		Then, the scheme is just a discretisation of each iterated integral:
		\begin{equs}
			U_{k}^{3,1}(t, v) = e^{-itk^2} v_k +   \sum_{ i =1}^3 e^{-itk^2} \frac{\Upsilon( T_i)(v)}{S(T_i)} (\Pi^{3,1}   T_i )(t)
		\end{equs}
		and the error made is of order $ \mathcal{O}(t^{3/2}) $ with an initial date that must in $ H^2 $ coming from the discretisation of $ (\Pi T_1)(t) $.
	\end{proof}
	\subsection{Stability and global error}
	The global convergence of the schemes presented in section \ref{Sec::4} is proved by using the Lax-Equivalence theorem which states that a scheme is globally convergent if and only if it is stable and consistent, which is obtained for SPDE's in \cite{O.P.W2022} and for PDE's is theorem 1 in \cite{T2012}.  Stability will be proved according to the definition \ref{def:stabitlity_and_consistancy} using standard tools from stochastic analysis built upon arguments typical in the analysis of numerical schemes for PDE, see for example section 3.4 or 4.4 of \cite{B2023}.  
	\subsubsection{Stability}
	\begin{proposition}\label{lem:stability}
		The numerical schemes \ref{scheme:schrodinger_O1/2} , \ref{scheme:schrodinger_O1} and \ref{scheme:manakov_O1/2} are stable.
	\end{proposition}
	\begin{remark}
		The proof of lemma \ref{lem:stability} for Schrödinger's equations and the Manakov system follows along the same lines as they contain the same nonlinearity and the stochastic parts are linear. We produce the proof for Scheme \ref{scheme:schrodinger_O1}, for which stability of the lower order Scheme \ref{scheme:schrodinger_O1/2} is a special case.  This proof also uses all of the ideas needed to prove the stability of Scheme \ref{scheme:manakov_O1/2} for the Manakov system.
	\end{remark}
	\begin{proof}[of stability for \ref{scheme:schrodinger_O1}]	
		First note that $ e^{it\lap} $ is an isometry on $ H^k $. We have for all $ v_1,v_2\in B_R(H^2)=\set{w\in H^2:\norm{w}<R} $ the existence of a constant depending on $ R $ such that
		\begin{equs}
			\label{eq:stability}
			\begin{aligned}
				\norm{S_t(v_1)-S_t(v_2)}_{k,p} &= \bigg|\bigg|e^{it\Delta}(v_1 - v_2) -it e^{it\Delta}\left(v^2_1\bar{v}_1 - v^2_2\bar{v}_2\right)\nonumber\\
				&\quad -i e^{it\Delta} (v_1 - v_2) \Phi W(t)  + e^{i t \Delta}(v_1 - v_2)\int_0^t s \lap\Phi dW(s) \nonumber\\
				&\quad +e^{i t \Delta} 2 \Psi_1(it\grad) (v_1 - v_2) \int_0^t s \grad \Phi dW(s)   \nonumber\\
				&\quad - \frac{1}{2}e^{it\Delta}(v_1 - v_2)\rbr{(\Phi W)^2(t) - t} \nonumber\\
				&\quad - t(2i+1)\rbr{v^2_1\bar{v}_1 - v^2_2\bar{v}_2}e^{it\lap}\Phi W(t) \nonumber \\
				&\quad +  i e^{it \Delta}\rbr{v_1\bar{v}_1^2 - v_2\bar{v}_2^2} \int_{0}^{t} \Phi \bar{W}(s)d s\nonumber\\
				&\quad + i e^{i t \Delta} (v_1 - v_2) \rbr{\frac16(\Phi W)^3(t) - \frac12t\text{Tr}(\Phi^2) \Phi W(t)}\bigg|\bigg|_{k,p}.
			\end{aligned}
		\end{equs}
		For the nonlinearity we first use the fact that the Schr\"odinger group is an isometry in $ H^s $ and perform the following estimate using Plancherel's theorem ,
		\begin{align*}
			&	\norm{v_1^2\bar{v}_1 - v_2^2\bar{v}_2}_{k,p}\\
			& = \norm{v_1^2\bar{v}_1 - v_1^2\bar{v}_2 + v^2_1\bar{v}_2 - v_2^2\bar{v}_2}_{k,p} \\
			&\leq \rbr{\norm{v^2_1(v_1-v_2)}_{k,p}+ \norm{v_2(v_1+v_2)(v_1 - v_2)}_{k,p}}\\
			&\leq  C_R\norm{v_1-v_2}_{k,p}.
		\end{align*}	
		To bound $ \Phi W(t) $ we note that it is Gaussian and thus bounded almost surely. Similarly, let $ I^f_n(t) $ be the $ n $th Weiner chaos with deterministic integrand $ f $, then it was shown by Borel in \cite{B1978} that the tail bound,
		\[\mbbP\rbr{\sup_{t\in[0,1]}\abs{I^f_n(t)}\geq K}\leq C_1e^{-C_2K^{2/n}},\] 
		holds. Hence for a fixed $ n $ by passing to the limit $K\rightarrow \infty $ we obtain almost sure boundedness. This takes care of the terms 
		\begin{equation*}
			\int_0^t s \lap\Phi dW(s), \quad (\Phi W)^2(t) - t \mathrm{Tr}(\Phi^2), \quad\frac16(\Phi W)^3(t) - \frac12t \mathrm{Tr}(\Phi^2) \Phi W(t).
		\end{equation*}
		That leaves
		\[\int_{0}^{t} \Phi \bar{W}(s)d s\leq t\sup_{s<t}|\Phi \bar{W}(s)|.\]
		Denote the supremum of all of the stochastic processes involved in the right hand side of \eqref{eq:stability} by $ W^* $ ($W^*_{\ell}$ for the $\ell$th step), then we deduce that there exists a constant $ C_{t, R, W^*} > 1$ such that 
		\begin{equation}\label{eq:stability2}
			\norm{S_t(v_1)-S_t(v_2)}_{k,p} \leq C_{t, R, W^*}\norm{v_1-v_2}_{k,p}.
		\end{equation}
		By induction we obtain
		\begin{equation}\label{eq:stability1}
			\norm{S^N_t(v_1)-S^N_t(v_2)}_{k,p}\leq\prod_{\ell=1}^NC_{t, R, W^*_\ell} \norm{v_1-v_2}_{k,p}.
		\end{equation}
		If we define $ X_i = \prod_{\ell=1}^iC_{R,t,W_i^* } $ then the process $ \set{X_i}_{i\in\mbbN}$ is a sub-martingale. Furthermore, by Markov's inequality and \eqref{eq:stability1} we have, for any $ n\in\mbbN $ such that $ n\leq N $,
		\begin{align*}
			\mbbP\rbr{\norm{S^n_t(v_1) - S^n_t(v_2)}_{k,p}\geq \eps}&\leq\frac{\mbbE\sqbr{\norm{S^n_t(v_1) - S^n_t(v_2)}_{k,p}}}{\eps}\\
			&\leq \frac{\mbbE\sqbr{X_n}\norm{v_1 - v_2}_{k,p}}{\eps}.
		\end{align*}
		Taking the supremum of both sides and using a martingale inequality we obtain,
		\begin{align*}
			\sup_{n\leq N}\mbbP\rbr{\norm{S^n_t(v_1) - S^n_t(v_2)}_{k,p}\geq \eps}&\leq \sup_{n\leq N}\frac{\mbbE\sqbr{X_n}\norm{v_1 - v_2}_{k,p}}{\eps}\\
			& \leq C_p\frac{\mbbE\sqbr{X_N}\norm{v_1 - v_2}_{k,p}}{\eps}.
		\end{align*}
		Taking the limit as $ \norm{v_1-v_2}_{k,p}\rightarrow 0 $ gives the result.
	\end{proof}
	\subsubsection{Global error}
	We now state the global error for the three schemes presented previously in the section. The proof follows from an induction argument on the form of the local error followed by Doob's maximal inequality for strong error and a result by Kloeden and Neuenkirch \cite{K.N2007} for the pathwise error. 
	\begin{theorem}\label{prop:global_error}
		The schemes \ref{scheme:schrodinger_O1/2} and \ref{scheme:manakov_O1/2} are globally convergent to order $ \mcO(t^{1/2}) $. The scheme \ref{scheme:schrodinger_O1} is globally convergent to order $ \mcO(t) $.
	\end{theorem}
	\begin{proof}
		We will use the notation defined in subsection \ref{ss::notation}. We wish to bound the global error $ e_N = S_t^{N}(v) - v(T) $, where. By adding zero we obtain 
		\begin{equation}\label{eq:global_err1}
			\norm{e_{n+1}}_{k,p}=\norm{S_t\rbr{v(t_n)} - v(t_{n+1})}_{k,p}+\norm{S^{n+1}_t(v)-S_t(v(t_n))}_{k,p}.
		\end{equation}
		The first term in \eqref{eq:global_err1} is the local error and it follows from \eqref{eq:stability2} in the stability analysis that there exists a constant $ B_{t,R,W^*_n }>0 $ such that
		\begin{equation}
			\norm{S_t(v(t_n)) - S^{n+1}_t(v)}_{k,p} \leq e^{B_{t,R,W^*_n }}\norm{e_n}_{k,p}.
		\end{equation}
		Hence it follows from the local error analysis that, for constants $ A_T $ and $B_{t,R,W^*_n} $ and assuming local error order $ \mcO(t^{1+\alpha}) $, we have
		\begin{equation}
			\norm{e_{n+1}}_{k,p}\leq A_Tt^{1+\alpha} + e^{B_{t,R,W^*_n}}\norm{e_n}_{k,p},\quad e_0 = 0.
		\end{equation}
		By induction we find, using $ t = T/n $,
		\begin{align}
			\norm{e_{n+1}}_{k,p}&\leq A_T Te^{nB_{t,R,W^*_n}} t^\alpha \\
			&=A_{T}e^{B_{T,R,W^*_n}}t^{\alpha}\\
			&=A_{T,R,W^*_n}t^{\alpha}.
		\end{align}
		Note that $ A_{T,R,X_n} $ is a sub-martingale in $ n $ and
		that loss of order $ \mcO(t) $ arising due to the accumulation of error throughout the iterations. It follows then by almost sure boundedness of the Weiner chaos and Doob's maximal inequality that
		\begin{equation}\label{eq:strong_err}
			\mbbE\rbr{\sup_{t_n \in \pi(0,T)}\norm{e_n}^p_{k,p}}\leq \mbbE[A_{T,R,X_N,p}]t^{p\alpha}.
		\end{equation} 
		We can move to pathwise convergence, again by noting that $ t=T/N $, and rewriting \ref{eq:strong_err} in terms of $ n $ then applying lemma 2.2 of \cite{K.N2007} to obtain, for any $ \eps>0 $
		\begin{equation*}
			\sup_{t_n \in \pi(0,T)}\norm{e_n}_{k,p}\leq \mbbE[A_{T,R,X_N,p}]t^{\alpha-\eps}.
		\end{equation*}
	\end{proof}

	\bibliography{lowreg_random} 
	\bibliographystyle{plainurl}

\end{document}